\documentclass[12pt]{amsart}
\usepackage[utf8]{inputenc}

\usepackage[centering, margin={0.8in, 1.0in}, includeheadfoot]{geometry}

\usepackage{mathrsfs}
\usepackage{amsfonts}
\usepackage{tikz}
\usepackage{amssymb,bm}
\usepackage{amsmath}
\usepackage{amsthm}
\usepackage{mathtools}
\usepackage{palatino}
\usepackage{tabularx}

\newtheorem{thm}{Theorem}[section]

\newtheorem{lem}[thm]{Lemma}
\newtheorem{cor}[thm]{Corollary}

\newtheorem{prop}[thm]{Proposition}
\newtheorem{rmk}[thm]{Remark}

\newcommand{\Tmq}{\operatorname{Tr}_{\mathbb{F}_{q^m}/\F}}
\newcommand{\Tqp}{\operatorname{Tr}_{\mathbb{F}_{q}/\mathbb{F}_p}}
\newcommand{\ord}{\operatorname{ord}}
\newcommand{\Tmp}{\operatorname{Tr}_{\mathbb{F}_{q^m}/\mathbb{F}_p}}
\newcommand{\F}{\mathbb{F}_{q}}
\newcommand{\Fr}{\mathbb{F}_{q^m}}
\newcommand{\rad}{\operatorname{Rad}}

\newcommand{\rmv}[1]{}

\begin{document}

\title{On the number of $N$-free elements with prescribed trace}

\author{Aleksandr Tuxanidy and Qiang Wang}

\address{School of Mathematics and Statistics, Carleton
University,
 1125 Colonel By Drive, Ottawa, Ontario, K1S 5B6,
Canada.} 

\email{AleksandrTuxanidyTor@cmail.carleton.ca, wang@math.carleton.ca}

 
\thanks{The research of Aleksandr Tuxanidy and Qiang Wang is partially supported by OGS and NSERC, respectively, of Canada.}

\begin{abstract}
 In this paper we derive a formula for the number of $N$-free elements over a finite field $\mathbb{F}_q$ with prescribed trace, in particular trace zero, in terms of Gaussian periods. 
 As a consequence, we derive a simple explicit formula for the number of primitive elements, in quartic extensions of Mersenne prime fields, having absolute trace zero. 
We also give a simple formula in the case when $Q = (q^m-1)/(q-1)$ is prime.
More generally, for a positive integer $N$ whose prime factors divide $Q$ and satisfy the so called semi-primitive condition, we give an explicit formula for the number of $N$-free elements with arbitrary trace.
In addition we show that if all the prime factors of $q-1$ divide $m$, then the number of primitive elements in $\mathbb{F}_{q^m}$, 
with prescribed non-zero trace, is uniformly distributed. Finally we explore the related number, $P_{q, m, N}(c)$, of elements in $\mathbb{F}_{q^m}$
with multiplicative order $N$ and having trace $c \in \F$. Let $N \mid q^m-1$ such that $L_Q \mid N$, where $L_Q$ is the largest factor of $q^m-1$ with the same radical as that of $Q$. 
We show there exists an element in $\mathbb{F}_{q^m}^*$ of (large) order $N$ with trace $0$ if and only if $m \neq 2$ and $(q,m) \neq (4,3)$. 
Moreover we derive an explicit formula for 
the number of elements in $\mathbb{F}_{p^4}$ with the corresponding large order $L_Q = 2(p+1)(p^2+1)$ and having absolute trace zero, where $p$ is a Mersenne prime.
\end{abstract}

 \maketitle

%%%%%%%%%%%%%%%%%%%%%%%%%%%%%%%%%%%%%5

\section{Introduction}

Let $q$ be the power of a prime number $p$ and let $\F$ be a finite field with $q$ elements. 
In 1992, Hansen and Mullen \cite{hansen} conjectured that, except for very few exceptions, there exist irreducible and primitive polynomials of degree $m$ over $\F$ with any prescribed coefficient respectively. 
This led to a great deal of work in the area, and both of these conjectures have since been resolved in the affirmative 
(see \cite{wan, ham} for irreducibles, as well as see the survey in \cite{cohen survey} and \cite{cohen and presern (2008)} for primitives). 

Particular interest has also been placed in deriving explicit formulas for the exact number of irreducible polynomials of degree $m$ over $\F$ 
with one or more prescribed coefficients (see for example \cite{carlitz, koma, kuzmin, kuzmin 91, yucas} and the survey \cite{cohen survey} or Section 3.5 by S. D. Cohen in the Handbook of finite fields \cite{handbook}).
Here it is worth mentioning the following beautiful formula due to Carlitz \cite{carlitz} describing the number of monic irreducible polynomials of degree $m$ with a prescribed trace coefficient
(the coefficient of $x^{m-1}$). Let $I_{q, m}(c)$ denote the number of monic irreducible polynomials of degree $m$ over $\F$ with trace $c$. Let $\mu$ be the M\"{o}bius function. 

\begin{thm}[{\bf Carlitz (1952)}]\label{thm: carlitz}
Let $q $ be a power of a prime $p$ and let $m \in \mathbb{N}$. Then for any non-zero element $c \in \F\setminus\{0\}$, the number of monic irreducible polynomials of degree $m$ over $\F$ and 
with trace $c$ is given by
$$
I_{q,m}(c \neq 0) = \dfrac{1}{qm} \sum_{\substack{d \mid m \\ p \nmid d}} \mu(d) q^{m/d} = \dfrac{I_{q,m} - I_{q,m}(0)}{q-1},
$$
where
$$
I_{q,m} = \dfrac{1}{m} \sum_{d \mid m} \mu(d) q^{m/d}
$$
is the number of irreducible polynomials of degree $m$ over $\F$. 
\end{thm}
Note that 
\begin{equation}\label{eqn:carlitz}
I_{q,m}(c)  = \dfrac{I_{q,m} - I_{q,m}(0)}{q-1},
\end{equation} 
is a constant for any $c \in \F^*$, and so $I_{q,m}(c)$ is said to be {\em uniformly distributed} for $c \in \F^*$. 
One of the results of this paper concerns an analogy to (\ref{eqn:carlitz}) for primitive polynomials in some special cases. We will return to this concept later.

A monic irreducible polynomial of degree $m$ over $\F$ is called
{\em primitive} if it has a primitive element of $\mathbb{F}_{q^m}$ as one of its roots.
There is a correspondence between the primitive elements in $\mathbb{F}_{q^m}$  and the primitive polynomials of degree $m$ over $\F$. 
In fact the number of primitive elements in $\mathbb{F}_{q^m}$ is $m$ times the number of primitive polynomials of degree $m$ over $\F$. 
In the case of the primitive polynomials of degree $m$, or equivalently of primitive elements in $\mathbb{F}_{q^m}$, things are more complicated. 
Most of the work on primitive polynomials with prescribed coefficients focus on the asymptotic analysis for their number and existence. 
For example, the following existence result was first due to Cohen (see also \cite{cohen and presern (2005)} for a more self-contained proof). 
%In attacking the trace problem for primitives one should first note that, as Cohen \cite{cohen0} showed 
%in the following theorem (see also \cite{cohen and presern (2005)} for a more self-contained proof) existence of such primitives is guaranteed 
%except for the trivial cases corresponding to trace zero. 
Denote with $\operatorname{Tr}_{\mathbb{F}_{q^m}/\F}$ the trace function from $\mathbb{F}_{q^m}$ onto $\F$. 

\begin{thm}[{\bf Cohen (1990)}]\label{thm: cohen}
 Let $q$ be a power of a prime and $m > 1$ be a positive integer, and let $c$ be an arbitrary element in $\F$. If $m = 2$ or $(q,m) = (4,3)$, further assume that $c \neq 0$. 
 Then there exists a primitive element $\xi$ of $\mathbb{F}_{q^m}$ with $\Tmq(\xi) = c$.
\end{thm}

One can see Section 4.2 by S. D. Cohen and the references therein, in the Handbook of Finite Fields \cite{handbook}. 
In fact, except for the trivial cases and those when all the primitive polynomials of degree $m$ are all the irreducibles of degree $m$
(i.e., when $q = 2$ and $m = \ell$ with $2^\ell-1$ a (Mersenne) prime) no explicit formulas are known to date.
In particular an analogue, for primitives, to the formula (\ref{eqn:carlitz}) due to Carlitz is unknown, including in any specific non-trivial case of $q,m$. 
The results of this paper are written in terms of the primitive elements in $\mathbb{F}_{q^m}$;
so we follow this convention from now on. Let $P_{q, m}$ be the number of primitive elements in $\mathbb{F}_{q^m}$. 
It is known that $P_{q, m} = \phi(q^m-1)$, where $\phi$ is the Euler function. For $c\in \F$, let $P_{q, m}(c)$ denote the number of primitive elements in $\mathbb{F}_{q^m}$ with trace $c$. 
Then as a corollary of Theorem~\ref{thm: uniformity} of this paper, the following result (analogous to (\ref{eqn:carlitz})) is proved. For a positive integer $n $, 
let us denote with $\rad(n)$ the product of all the distinct prime factors of $n$. By convention $\rad(1) = 1$.

\medskip
\noindent {\bf Corollary of Theorem~\ref{thm: uniformity}} Let $q$ be a prime power and $m$ be a multiple of $\rad(q-1)$. Then, for $c\neq 0$, we have 
\begin{equation*}
P_{q,m}(c)  = \dfrac{P_{q,m} - P_{q,m}(0)}{q-1},
\end{equation*}

Although the formula above corresponds to primitive elements and hence primitive polynomials, we will however consider, in the sections that follow, the more general concept of 
an element of $\mathbb{F}_{q^m}$ being $N$-free, for a positive divisor $N$
of $q^m-1$. Let us first fix the following notations and definitions. 

{\bf Notations:}
In what follows we let $q = p^s$ be a power of a prime number $p$, $m$ be a positive integer, $Q = (q^m-1)/(q-1)$, and  $\alpha$ be a primitive element of $\Fr$.
For a positive divisor $N$ of $q^m-1$, we say that a non-zero element $\xi \in \Fr^*$ is {\em $N$-free} 
if, for any $d \mid N$, $\xi = \gamma^d$, $\gamma \in \Fr$, implies $d = 1$. Equivalently, $\xi$ is $N$-free if and only if $\xi = \alpha^k$ for some integer $k$ that is coprime to $N$.
Note that the definition of $N$-free is independent of the choice of the primitive element $\alpha$. Furthermore, 
for an element $c \in \F$, we denote with $Z_{q, m,N}(c)$ the number of $N$-free elements $\xi$ in $\Fr^*$ such that $\Tmq(\xi) = c$. Moreover we let $P_{q,m,N}(c)$ be the number of non-zero elements $\zeta$ in $\Fr^*$ with multiplicative order $N$ and satisfying $\Tmq(\zeta) = c$.
In particular we note that
 \begin{equation*}
 Z_{q,m,q^m-1}(c)= P_{q,m,q^m-1}(c) = P_{q, m}(c)
 \end{equation*}
 is the number of primitive elements $\xi$ in $\Fr$ such that $\Tmq(\xi) = c$. For an integer $k$ and $N \mid q^m-1$, denote
$$
\Delta_k(N) := \sum_{d \mid N} \mu(d) \eta_k^{(d, q^m)},
$$
where in the sum $\eta_k^{(d,q^m)}$ is the $k$-th Gaussian period of type $(d,q^m)$ (we refer the reader to Section 2 for its definition).
Note that the value of $\Delta_k(N)$ depends only on the square-free part of $N$.  
%As we shall see later on, there is a special reason behind the choice of the notation $\Delta$.

Previously Cohen and Pre\u{s}ern \cite{cohen and presern (2005)} derived a formula for $Z_{q,m,N}(c)$ in terms of Gaussian sums (see Lemma 2.2 there). From this they were 
able to obtain lower bounds,  through various assisting sieving inequalities,  thus proving Theorem \ref{thm: cohen}  in a more self-contained fashion than previously done in 
\cite{cohen0}. However as it was perhaps beyond the scope of their work, and except for their Corollary 2.3 where they give an explicit formula for $Z_{q,m,N}(c)$ in a few special cases of 
Corollary  \ref{thm: semi delta}  and Corollary \ref{cor: semidelta uniform} below, 
their results were mainly constrained to lower bounds and existence results. It is interesting to note that in the case of trace zero, as they showed in their Lemma 2.1, 
there is the connection between primitives with trace zero and $Q$-free elements with trace zero: $ Z_{q,m,q^m-1}(0) = \Theta(K) Z_{q,m,Q}(0)$. Here $K$ is the part of $q^m-1$ that is coprime to $Q$, and
$\Theta(K) = \phi(K)/K$ is the proportion of primitive $K$-th roots of unity among the $K$-th roots. 
But more generally, as we show here thorough our calculations, a lemma due to Ding and Yang \cite{ding} (see Lemma \ref{lem: ding} here) 
implies that something similar holds in general for any divisor $N$ of $q^m-1$: 
$Z_{q,m,N} (0) = \Theta(K) Z_{q,m, \gcd(Q, N)}(0)$, where $K$ is now the part of $N$ that is coprime to $Q$. 
See the following theorem, proved later in Section 4.1.

\begin{thm}\label{lem: zero trace}
Let $q$ be a power of a prime and $m$ be a positive integer. Let $N \mid q^m-1$ and $K_Q$ be the largest divisor of $N$ that is coprime to $Q = (q^m-1)/(q-1)$. Then
 $$
  Z_{q,m,N}(0) = \dfrac{(q-1)\phi(K_Q)}{qK_Q} \left( \dfrac{Q}{\gcd(Q, N)} \phi(\gcd(Q, N)) + \Delta_0(\gcd(Q, N)) \right). 
 $$
\end{thm}

Note that obtaining the value of $Z_{q,m, N}(0)$ boils down to computing $\Delta_0(\gcd(Q,N))$. Since Gaussian sums and hence periods are known in only very few cases, obtaining images of 
$\Delta_0$ may be quite hard in general.
But by using known results on periods we can clearly obtain some explicit expressions. For instance
we obtain the following two direct consequences when $\gcd(Q, N)=1$. 

\begin{cor}\label{thm: N coprime to Q}
Let $q$ be a power of a prime, let $m \in \mathbb{N}$ and let $N \mid q-1$ such that $N$ is coprime to $(q^m-1)/(q-1)$.
Then the number of $N$-free elements $\xi \in \mathbb{F}_{q^m}$ with $\operatorname{Tr}_{\mathbb{F}_{q^m}/\F}(\xi) = 0$ is given by 
  $$
 Z_{q,m,N}(0) = \dfrac{\phi(N)}{N} \left( q^{m-1} - 1 \right).
  $$
 \end{cor}
 
 Setting $N = 1$ above one obtains the well-known number, $q^{m-1}-1 $, of non-zero elements lying in the kernel of the trace map. 

\begin{cor}\label{thm: Q prime with trace 0}
  Let $q$ be a power of a prime $p$ and assume that $Q = (q^{\ell}-1)/(q-1)$ is prime for some prime $\ell$. Then the number of primitive elements $\xi \in \mathbb{F}_{q^{\ell}}$ 
  satisfying $\operatorname{Tr}_{\mathbb{F}_{q^{\ell}}/\F}(\xi) = 0$ is given by
  $$
  Z_{q,\ell,q^{\ell}-1}(0) = \begin{cases}
                    \phi(q^{\ell}-1)/q & \mbox{ if } \ell \neq p;\\
                    \phi(q^{\ell}-1)/q - \phi(q-1)  & \mbox{ otherwise.}
                   \end{cases}
$$
 \end{cor}
 Since quadratic and cubic Gaussian periods are known as well, we also immediately obtain Corollaries \ref{thm: N = 2} and \ref{thm: N = 3}.
 These two correspond to the cases when $\gcd(Q, N)$ is a power of 2 and 3, respectively.

 \begin{cor}\label{thm: N = 2}
  Let $q = p^s$, $Q = (q^m-1)/(q-1)$,  $N \mid q^m-1$ such that $\gcd(Q, N) = 2^n$ for some $n \geq 1$, and $K_Q$ be the largest odd divisor of $N$. Then
  $$
  \dfrac{2 q K_Q }{(q-1)\phi(K_Q)}  Z_{q,m,N}(0) = \begin{cases}
                   Q - 1 + (-1)^{sm}q^{m/2}   & \mbox{ if } p \equiv 1 \pmod{4};\\
                  Q - 1 +  \left(- \sqrt{-1} \right)^{sm}q^{m/2}   & \mbox{ if } p \equiv 3 \pmod{4}.
                 \end{cases}
$$
 \end{cor}

 \begin{cor}\label{thm: N = 3}
   Let $p \equiv 1 \pmod{3}$, 
   let $q = p^s$,  $Q = (q^m-1)/(q-1)$, $N \mid q^m-1$ such that $\gcd(N,Q) = 3^n$ for some $n \geq 1$, and $K_Q$ be the 
   largest divisor of $N$ with $3 \nmid K_Q$.
   Let $c \in \mathbb{Z}$ with $c \equiv 1 \pmod{3}$ and $p \nmid c$ be a solution to the equation
   $4q^{m/3} = c^2 + 27d^2$ with $d \in \mathbb{Z}$.  
   Then
   $$
    Z_{q,m,N}(0) = \dfrac{(q-1)\phi(K_Q)}{qK_Q} \left( \dfrac{2Q - 2 - cq^{m/3}}{3}  \right).
   $$
   \end{cor}

Thanks in part to well-known explicit expressions for the Gaussian periods in the so called semi-primitive case (see Lemma \ref{lem: semiprimitive}), we obtain the following result. 

\begin{cor}\label{thm: semi delta}
    Let $sm$ be even with $m > 1$, $q = p^s$ be a power of a prime $p$, and $Q = (q^m-1)/(q-1)$. Let $N \mid q^m-1$ be such that $n := \gcd(Q, N) > 1$
    is not a power of $2$. Further assume there exists a positive integer $j$ such that $p^j \equiv -1 \pmod{\ell}$ for every prime divisor $\ell \geq 3$ of $n$, and that $j$ is the least such. 
    Define $\gamma = sm/2j$.
    Let $K_Q$ be the part of $N$ that is coprime to $Q$. Let $\eta_0^{(2,q^m)}$ be as in Lemma \ref{lem: period N = 2}. Then the number of $N$-free elements $\xi \in \Fr$ with $\Tmq(\xi) = 0$ is given by 
    $$
    Z_{q,m, N}(0) = \dfrac{(q-1)\phi(K_Q)}{qK_Q} \left( \dfrac{Q}{n} \phi(n) + \Delta_0(n) \right),
    $$
    where the value of $\Delta_0(n)$ is given in what follows.
    
    (a) If $\gamma$ and $p$ are odd, $n$ is even and $(p^j+1)/2$ is odd, then
    $$
    \Delta_0(n) = - \eta_0^{(2,q^m)} - \left( 1 + q^{m/2}   \right)\left( \dfrac{1}{2} + \dfrac{ \phi\left( n \right) }{n}  \right).
    $$
    
    (b) In all other cases,
    $$
    \Delta_0(n) = -\epsilon_2 \cdot \left(\dfrac{(-1)^{\gamma }q^{m/2} + 1}{2} + \eta_0^{(2,q^m)}\right) + \dfrac{(-1)^{\gamma} q^{m/2} -1}{n} \phi(n),
    $$
    where
    $$
    \epsilon_2 = \begin{cases}
                  1 & \mbox{ if } n \mbox{ is even;}\\
                  0 & \mbox{ otherwise.} 
                 \end{cases}
$$
   \end{cor}

It is a simple matter to show that in the case of quadratic extensions ($m=2$) no primitive element of $\mathbb{F}_{q^2}$ with trace zero, in $\F$, exists.
In fact,  this case falls under the more general category below of Proposition \ref{prop: mersenne}, for which we are able to, in this paper, obtain the corresponding formula for the zero trace. 

\begin{prop}\label{prop: mersenne}
 Let $q = p^s$ be a power of a prime number $p$ and let $m > 1$ be an integer. Then there exists a positive integer $j$ such that
 $$
 p^j \equiv -1 \pmod{\rad\left( \frac{q^m-1}{q-1}  \right)}
 $$
 if and only if $m = 2$, or $q = p$ is a Mersenne prime and $m = 4$. The latter case holds precisely for every $j = 2k$ with $k \geq 1$ odd.
\end{prop}

Recall that a Mersenne prime $M_{\ell}$ is of the form
$M_{\ell} = 2^{\ell} - 1$ for some prime $\ell$. Usually the world's record of the largest prime is broken by a Mersenne prime, and,
although only 48 such primes have been discovered thus far (see the Great Internet Mersenne Prime Search (GIMPS) available online) it is a well-known conjecture that there exist infinitely many of them.
They appear in various areas of number theory and finite fields, including in the Great Trinomial Hunt \cite{GTH}, an ongoing project for the search of primitive trinomials (i.e., primitive 
polynomials with exactly
three non-zero terms) over $\mathbb{F}_2$ with degree the ``exponent'' $\ell$ of a Mersenne prime $M_{\ell}$.

We obtain the following simple formula for the number of primitive elements, with absolute trace zero, in quartic extensions of Mersenne prime fields. 
%Let $\phi$ be the Euler's totient function. In general, there are $\phi(q-1)$ primitive elements in a finite field $\F$. 

\begin{cor}\label{thm: tuxanidy}
 Let $p$ be a Mersenne prime. Then the number of primitive elements $\xi$ in $\mathbb{F}_{p^4}$ satisfying $\operatorname{Tr}_{\mathbb{F}_{p^4}/\mathbb{F}_p}(\xi) = 0$ is given by
 $$
 \dfrac{1}{p}\left(\phi\left( p^4-1 \right) - \phi \left( \dfrac{p^4-1}{p+1}  \right) \right).
 $$
\end{cor}

 In Section 5 we finally return to the concept of uniformity, already met in Theorem \ref{thm: carlitz}, now for $N$-free elements; in particular, for primitive elements. 
 Although it is easy to find examples of $q,m,N$ for which $Z_{q,m,N}(c)$ does not behave uniformly for $c \in \F^*$ (and indeed, when $N = q^m-1$ is fixed as well) 
 it is of special interest to find and classify instances of $q,m,N$ for which 
 $Z_{q,m,N}(c\neq 0)$ does. One of the obvious reasons being that, in this case, in order to obtain the number of $N$-free elements with a prescribed non-zero trace, it would be enough to find the corresponding 
 number for the zero trace. The following theorem gives a sufficient criteria for this to happen, but we ask the interested able reader to characterize all such instances of $q,m,N$.
 
 \begin{thm}\label{thm: uniformity}
  Let $q$ be a power of a prime and $N$ be a positive divisor of $q^m-1$. If every prime divisor of $N$ divides $Q = (q^m-1)/(q-1)$, 
  then, for every element $c \in \F \setminus\{0\}$, the number of $N$-free elements $\xi \in \Fr$ satisfying $\Tmq(\xi) = c$ is given by
  \begin{equation*}
  Z_{q,m,N}(c \neq 0)= \dfrac{\frac{q^m-1}{N} \phi(N) - Z_{q,m,N}(0)}{q-1}.
  \end{equation*}
 \end{thm}

 In particular, setting $N = q^m-1$, we obtain that $Z_{q,m,q^m-1}(c)$ is a constant independent of $c \in \F^*$ whenever the radical (the product of all the distinct prime divisors) 
 of $Q$ is the same as that of $q^m-1$. This occurs exactly when all the prime factors of $q-1$ divide $m$; see Corollary \ref{cor: unif} for this. 
Thus we obtain the following immediate consequence to Theorem \ref{thm: uniformity} and Corollary \ref{thm: semi delta}.
 
 \begin{cor}\label{cor: semidelta uniform}
  Assume that $q,r,N,$ satisfy the assumptions of Corollary  \ref{thm: semi delta} and further assume that $\rad(N) \mid Q$. Let $\eta_0^{(2,q^m)}$ be as in Lemma \ref{lem: period N = 2}. Then for any non-zero $c \in \F^*$, the number of $N$-free elements
  $\xi \in \Fr$ with $\Tmq(\xi) = c$ is given in what follows. 
  
  (a) If $\gamma$ and $p$ are odd, $N$ is even and $(p^j+1)/2$ is odd, then
  \begin{equation*}
  Z_{q,m,N}(c \neq 0)= \dfrac{1}{q} \left( \eta_0^{(2,q^m)} + \left( 1 + q^{m/2} \right)\left( \dfrac{1}{2} +  \dfrac{\phi(N)}{N} \right)   +  \phi(N)\left(  \dfrac{q Q}{N} - 1  \right)   \right).
  \end{equation*}
  
  (b) In all other cases,
  $$
  Z_{q,m, N}(c\neq 0) = \dfrac{\phi(N)}{qN} \left( q^m + (-1)^{\gamma +1} q^{m/2} + qQ - N + \epsilon_2 \cdot \left(\dfrac{(-1)^{\gamma }q^{m/2} + 1}{2} + \eta_0^{(2,q^m)}\right)   \right),
  $$
  where
 $$
    \epsilon_2 = \begin{cases}
                  1 & \mbox{ if } N \mbox{ is even;}\\
                  0 & \mbox{ otherwise.} 
                 \end{cases}
$$
 \end{cor}

 As a consequence of Theorem \ref{thm: uniformity} one obtains the following interesting property of the sum $\Delta_0(N)$ for $N \mid q^m-1$ such that $\rad(N) \mid Q = (q^m-1)/(q-1)$.
 It is the constant difference between the number of $N$-free elements with zero and non-zero traces in $\mathbb{F}_t$, for any subfield $\mathbb{F}_t$ of $\F$.
 
 \begin{cor}\label{cor: uniformity 1}
Let $N \mid q^m-1$ such that every prime divisor of $N$ divides $(q^m-1)/(q-1)$. Then for every subfield $\mathbb{F}_t$ of $\F$ and every $c_t \in \mathbb{F}_t^*$, we have
  $$ 
  \Delta_0(N) = Z_{t, m[\F:\mathbb{F}_t], N}(0) - Z_{t, m[\F:\mathbb{F}_t], N}(c_t \neq 0).
  $$
 Furthermore, if $t \neq q$, then
  $$
  \Delta_0(N) = \dfrac{q Z_{q, m, N}(0) - t Z_{t, m[\F:\mathbb{F}_t], N}(0)}{q - t}.
  $$
  \end{cor}
  
  Although the paper is primarily concerned with the number $Z_{q,m,N}(c)$, we briefly consider in Section 6 the seemingly closely related number, $P_{q,m,N}(c)$,
  of elements with order $N$ having a prescribed trace $c$. There we derive a general formula for $P_{q,m,N}(c)$ (see Lemma \ref{lem: relation}) and show its relation to $Z_{q,m,N}(c)$ as well as Hamming weights
  of specific codewords in irreducible cyclic codes.
  Let $L_Q$ be the largest divisor
  of $q^m-1$ with the same radical as that of $Q$. We also show, in Lemma \ref{lem: relation},
  that if $N \mid q^m-1$ is such that $L_Q \mid N$, then the following simple relation holds: $Z_{q,m,N}(0) = \frac{q^m-1}{N} P_{q,m,N}(0)$.
  We believe it should not be too difficult to generalize this even further for arbitrary $N$, but we leave this to the interested reader. 
  As a consequence of this and of Cohen's result (see Theorem \ref{thm: cohen} above) we show in Theorem \ref{thm: existence result} that 
  there exists an element in $\mathbb{F}_{q^m}^*$ of order $N$ (with $L_Q \mid N$)
having trace $0$ if and only if $m \neq 2$ and $(q,m) \neq (4,3)$.
   In addition Lemma \ref{lem: relation} together with Corollary \ref{lem: zero trace} yields, in
  Theorem \ref{thm: order DQ and mersenne}, an explicit expression for the number of elements of order $2(p+1)(p^2+1)$, 
  in quartic extensions of Mersenne prime fields $\mathbb{F}_p$, having absolute trace zero. 

 The rest of the paper goes as follows. In Section 2 we go over some preliminary concepts which will be of use in further sections. 
 In Section 3 we derive a formula for $Z_{q,m,N}(c)$ in terms of Gaussian periods (see Lemma \ref{lem: Z formula}). 
 In Section 4 we specifically consider the case of the zero trace and simplify our formula, in Subsection 4.1, with the use of a lemma due to Ding and Yang \cite{ding}. 
 Then in Subsection 4.2 we prove Corollaries \ref{thm: semi delta} and \ref{thm: tuxanidy}. In Section 5 we give a sufficient criteria for uniformity to occur (see Theorem \ref{thm: uniformity}),
 as well as 
 some other related results. Then in Section 6 we focus our attention to the number $P_{q,m,N}(c)$ and give some other related results. In particular, we obtain that  the number of non-zero elements in $\mathbb{F}_{p^4}^*$ with the corresponding large order $2(p+1)(p^2+1)$ and having absolute trace zero is $2\phi(p^2+1)$, where $p$ is a Mersenne prime (see Theorem \ref{thm: order DQ and mersenne}).
 Finally in Appendix we include a table of data corresponding to Corollary \ref{thm: tuxanidy}, giving the number of primitive elements in quartic extensions of Mersenne prime fields, 
 with absolute trace zero,
 for the first ten Mersenne primes.

\section{Preliminaries}

In this section we go over some preliminary concepts which will be of use in further sections. As before, we let $q = p^s$ be a power of a prime number $p$, let $\F$ be the finite field
with $q$ elements and  $\Fr$ be the degree-$m$ extension of $\F$.
The following concepts and definitions are well-known and may be found for example in Chapter 5 of \cite{lidl} and in \cite{ding}.
Now let $\chi_q, \chi_{q^m},$ be the canonical additive characters of $\F, \Fr,$ respectively, defined by 
$\chi_q(x) = e^{2\pi\sqrt{-1}\Tqp(x)/p}$ for $x \in \F$, and $\chi_{q^m} = \chi_q \circ \Tmq$. By the transitivity of the trace function, $\chi_{q^m}(z) = e^{2\pi\sqrt{-1}\Tmp(z)/p}$ for $z \in \Fr$. 
Denote with $ \chi_a^{(q^m)}$ the additive character of $\Fr$ corresponding to $a \in \Fr$; that is $ \chi_a^{(q^m)}(z) = \chi_{q^m}(az)$ for any $z \in \Fr$.
Clearly $\chi_1^{(q^m)} = \chi_{q^m}$. The following {\em orthogonality relation} will be of use.
\begin{equation}\label{eqn: orthogonality}
\sum_{a \in \F}\chi_q(ax) = \begin{cases}
                             q & \mbox{ if } x = 0;\\
                             0 & \mbox{ if } x \in \F^*.
                            \end{cases}
\end{equation}

Let $\alpha$ be a primitive element of $\Fr$. For a divisor $N$ of $q^m-1$, let $\psi_N$ be a 
multiplicative character of $\Fr$ of order $N$. That is, $\psi_N$ is defined by $\psi_N(\alpha^v) = e^{2\pi\sqrt{-1}jv/N}$ for some integer $j$ that is coprime to $N$.

The {\em Gaussian sums} of order $N$ are given by
$$
G_{q^m}\left(\psi_N,  \chi_a^{(q^m)}\right) = \sum_{\beta \in \Fr^*} \psi_N(\beta) \chi_{q^m}(a \beta).
$$
We denote $G_{q^m}(\psi_N) := G_{q^m}(\psi_N, \chi_{q^m})$. Note that if $a \neq 0$, then $G_{q^m}(\psi_N,  \chi_a^{(q^m)}) = \overline{\psi_N(a)} G_{q^m}(\psi_N)$ (Theorem 5.12 (i), \cite{lidl}).

For $N \mid q^m-1$,
the {\em cyclotomic classes} of $\Fr^*$ of type $(N,q^m)$ are defined by $C_k^{(N,q^m)} = \alpha^k \left< \alpha^N \right>$, where $k \in \mathbb{Z}$. 
Clearly $C_k^{(N,q^m)} = C_0^{(N,q^m)}$ whenever $k \equiv 0 \pmod{N}$.
Then the {\em Gaussian periods} of type $(N,q^m)$ are given by
$$
\eta_k^{(N,q^m)} = \sum_{x \in C_k^{(N,q^m)}} \chi_{q^m}(x).
$$
The Gaussian sums are the discrete Fourier transforms of the Gaussian periods and hence the two are related by the equation
\begin{align}\label{eqn: periods and gauss sums}
 \eta_k^{(N,q^m)} &= \dfrac{1}{N} \sum_{j=0}^{N-1} \sum_{x \in \Fr^*} \chi_{q^m} \left( \alpha^k x  \right) \psi_N^j (x)
 = \dfrac{1}{N} \sum_{j=0}^{N-1}\overline{\psi_N^j\left( \alpha^k  \right)}G_{q^m}\left( \psi_N^j \right) \\
 &= \dfrac{1}{N} \left( -1 + \sum_{j=1}^{N-1}\overline{\psi_N^j\left( \alpha^k  \right)}G_{q^m}\left( \psi_N^j \right) \right) \nonumber,
\end{align}
where $\psi_N$ is a multiplicative character of $\Fr$ with order $N$ (see Equation (9) in \cite{ding}).

In their study of Hamming weights of irreducible cyclic codes, Ding and Yang \cite{ding} recently obtained the following result regarding cyclotomic classes. 
 
\begin{lem}[{\bf Lemma 5, \cite{ding}}]\label{lem: ding}
 Let $N$ be a positive divisor $q^m-1$ and let $k \in \mathbb{Z}$. We have the following multiset equality:
 $$
 \left\{ ax \ : \ a \in \F^*,\ x \in C_k^{(N,q^m)}   \right\} = \dfrac{(q-1)\gcd\left(Q,  N\right)}{N} \ast C_k^{\left( \gcd\left(Q, N\right), q^m  \right)},
 $$
 where the right hand side denotes the multiset in which each element in the set $ C_k^{\left( \gcd\left(Q, N\right), q^m  \right)}$ appears in the multiset 
 with multiplicity $\frac{(q-1)\gcd\left(Q, N\right)}{N}$.
\end{lem}

A consequence to the above is the following. 

\begin{lem}\label{lem: ding2}
 Let $N \mid q^m-1$ and let $k \in \mathbb{Z}$. Then
 $$
 \sum_{i=0}^{q-2} \eta_{Qi + k}^{(N,q^m)} = \dfrac{(q-1)\gcd\left(Q, N\right)}{N} \eta_k^{\left(\gcd\left(Q, N\right), q^m  \right)}.
 $$
\end{lem}

\begin{proof}
 By the definition of Gaussian periods and by Lemma \ref{lem: ding},
 \begin{align*}
  \sum_{i=0}^{q-2} \eta_{Qi + k}^{(N,q^m)} &= \sum_{i=0}^{q-2} \sum_{x \in C_0^{(N,q^m)}} \chi_{q^m} \left( \alpha^{Qi + k} x  \right) = \sum_{i=0}^{q-2} \sum_{x \in C_k^{(N,q^m)}}\chi_{q^m} \left( \alpha^{Qi } x  \right)\\
  &= \sum_{a \in \F^*} \sum_{x \in C_k^{(N,q^m)}}\chi_{q^m} \left( a x  \right)
  = \dfrac{(q-1)\gcd\left(Q, N\right)}{N} \sum_{x \in C_k^{(\gcd\left(Q, N\right), q^m)}} \chi_{q^m} (x)   \\
  &= \dfrac{(q-1)\gcd\left(Q, N\right)}{N} \eta_k^{(\gcd\left(Q, N\right), q^m)}.
 \end{align*}
\end{proof}

\begin{rmk}
 It is known that $\eta_k^{(N,q^m)} \in \mathbb{Z}$ whenever $N \mid Q$ (see Theorem 13 (i) in \cite{ding}).
\end{rmk}

The following results about Gaussian periods are well known and may be found for example
in \cite{ding}. We only give the $0$-th Gaussian periods as these will be of greater interest to us in the sections that follow.
For the other cases we refer the interested reader to \cite{ding}. First it is easy to show that $\eta_k^{(1,q^m)} = -1$ for any $k \in \mathbb{Z}$ and hence $\Delta_k(1) = -1$.

\begin{lem}\label{lem: period N = 2}
When $N=2$, the $0$-th Gaussian periods are given by the following:
\begin{align*}
\eta_0^{(2,q^m)} &= \begin{cases}
                    \frac{-1 + (-1)^{sm-1}q^{m/2}}{2} & \mbox{ if } p \equiv 1 \pmod{4};\\
                    \frac{-1 + (-1)^{sm-1} \left( \sqrt{-1}  \right)^{sm} q^{m/2}}{2} & \mbox{ if } p \equiv 3 \pmod{4}.
                  \end{cases}\\
          \end{align*}
\end{lem}

In the case when $N = 3$ we only give a particular instance although the 
results in other cases are also known.  

\begin{lem}\label{lem: period N = 3}
 Let $N = 3$, let $sm \equiv 0 \pmod{3}$, let $p \equiv 1 \pmod{3}$, and let $c,d,$ be the unique (up to sign) solutions 
 to the equation $4p^{sm/3} = c^2 + 27 d^2$ with $c \equiv 1 \pmod{3}$ and $p \nmid c$.
 Then
 $$
 \eta_0^{(3,q^m)} = \dfrac{-1 + cq^{m/3}}{3}.
 $$
\end{lem}

The Gaussian periods in the so called {\em semi-primitive} case are known as well
and are given in the following lemma. See \cite{ding}.

\begin{lem}\label{lem: semiprimitive}
 Assume that $N > 2$ and there exists a positive integer $j$ such that $p^j \equiv -1 \pmod{N}$
 and that $j$ is the least such. Let $r = p^{2j \gamma}$ for some integer $\gamma$.
 
 (a) If $\gamma, p$ and $(p^j+1)/N$ are all odd, then
 \begin{align*}
  \eta_{0}^{(N,r)} = -\dfrac{r^{1/2} + 1}{N}.
 \end{align*}
 
 (b) In all other cases,
 $$
 \eta_0^{(N,r)} = \dfrac{(-1)^{\gamma + 1}(N-1)r^{1/2} - 1}{N}.
 $$
\end{lem}

The following lemma will be useful as well.

\begin{lem}\label{lem: period Q}
 Let $q$ be a power of a prime $p$ and let $Q = (q^m-1)/(q-1)$. Then
 $$
 \eta_0^{(Q, q^m)} = \begin{cases}
                    -1 & \mbox{ if } p \nmid m;\\
                    q-1 & \mbox{ otherwise.}
                   \end{cases}
$$
\end{lem}

\begin{proof}
 By definition,
 $$
  \eta_0^{(Q,q^m)} = \sum_{x \in \left< \alpha^Q \right>} \chi_{q^m}(x) = \sum_{x \in \F^*} \chi_{q^m}(x) = \sum_{x \in \F^*} \chi_q\left(\Tmq(x)  \right) = \sum_{x \in \F^*}\chi_q(mx).
$$
Now the result follows by (\ref{eqn: orthogonality}).
\end{proof}

\section{A formula for $Z_{q,m, N}$}

In this section we derive, in terms of Gaussian periods, a formula for the number of $N$-free elements with prescribed trace (see Lemma \ref{lem: Z formula}). 
Note that Cohen and Pre\u{s}ern \cite{cohen and presern (2005)} already did so in terms of Gaussian sums (see their Lemma 2.2).
However by the means of Gaussian periods we will be able to apply Ding-Yang lemmas (Lemma \ref{lem: ding} and \ref{lem: ding2}) thus obtaining, 
for the case of the zero trace,
the simplified version of Theorem \ref{lem: zero trace} and the fact that $Z_{q,m,N}(0) = \Theta(K) Z_{q,m,\gcd(Q,N)}(0)$ already mentioned in the introduction. Recall that here
$K$ is the part of $N$ that is coprime to $Q$,
and $\Theta(K) = \phi(K)/ K$. 

The following characteristic function for $N$-free elements, due to 
Vinogradov, is typically used in works on the topic. See for instance 
\cite{cohen0, cohen survey, cohen and presern (2005), cohen and presern (2008)} and the references therein.

\begin{prop}[{\bf Vinogradov}]
 Let $N$ be a positive divisor of $q^m-1$ and let $\xi \in \Fr^*$. Then 
 $$
 \dfrac{\phi(N)}{N}\sum_{d \mid N} \dfrac{\mu(d)}{\phi(d)} \sum_{\ord(\psi) = d} \psi\left( \xi \right) = \begin{cases}
                                                                                                          1 & \mbox{ if } \xi \mbox{ is $N$-free};\\
                                                                                                          0 & \mbox{ otherwise,}
                                                                                                         \end{cases}
                                                                                                         $$
where in the inner sum $\psi$ runs through all the multiplicative characters of $\Fr$ with order $d$.
\end{prop}

We will however consider the following apparently simpler form of the characteristic function.

\begin{lem}\label{lem: characteristic}
 Let $N$ be a positive divisor of $q^m-1$ and let $\xi \in \Fr^*$. For each positive divisor $d$ of $N$, fix a multiplicative character $\psi_d$  
 of $\Fr$ with order $d$. Then
 $$
 \sum_{d \mid N} \dfrac{\mu(d)}{d} \sum_{j=0}^{d-1} \psi_d^j \left(  \xi \right) = \begin{cases}
                                                                                                          1 & \mbox{ if } \xi \mbox{ is $N$-free};\\
                                                                                                          0 & \mbox{ otherwise.}
                                                                                   \end{cases}
$$                                                                                                        

\begin{proof}
 Let $\alpha$ be primitive in $\Fr$. Then $\xi = \alpha^k$ for some integer $k$. Note that
 $$
 \dfrac{1}{d}\sum_{j=0}^{d-1} e^{2\pi\sqrt{-1} jk/d} = \begin{cases}
                                            1 & \mbox{ if } d \mid k;\\
                                            0 & \mbox{ otherwise}.
                                           \end{cases}
$$
Hence
\begin{align*}
 \sum_{d \mid N} \dfrac{\mu(d)}{d} \sum_{j=0}^{d-1} \psi_d^j \left(  \alpha^k \right) &= \sum_{d \mid N} \dfrac{\mu(d)}{d} \sum_{j=0}^{d-1} e^{2\pi\sqrt{-1} jk/d} = \sum_{d \mid \gcd(N,k)} \mu(d) = \begin{cases}
                                                                                                                                                                                                       1 & \mbox{ if } \gcd(N,k) = 1;\\
                                                                                                                                                                                                      0 & \mbox{ otherwise.}
                                                                                                                                                                                                      \end{cases}
                                                                                                                                                                                                      \end{align*}
                                                                                                                                                                                                      \end{proof}
\end{lem}

Recall that for a positive divisor $N$ of $q^m-1$, an integer $k$, and $c \in \F$, we denote
\begin{align*}
\Delta_k(N) &= \sum_{d \mid N} \mu(d) \eta_k^{(d,q^m)}.
\end{align*}

The following proposition highlights some of the basic properties of $\Delta_k$.

\begin{prop}\label{prop: identities}
 Let $N \mid q^m-1$ and let $k \in \mathbb{Z}$. Then we have the following three identities:
 \begin{align*}
  \Delta_{Nk}(N) &= \Delta_0(N);\\
  \sum_{i=0}^{N-1} \Delta_i(N) &= -\phi(N); \text{ and}\\
   \Delta_k(N) &= \sum_{\substack{i=1 \\ \gcd(N, i-k) = 1}}^{q^m-1} \chi_{q^m} \left( \alpha^i \right).
 \end{align*}
 \end{prop}
 
 \begin{proof}
 The first identity follows from the fact that the sequence of periods of type $(N,q^m)$ has period $N$, i.e., $\eta_{k}^{(N,q^m)} = \eta_0^{(N,q^m)}$ whenever $k \equiv 0 \pmod{N}$. 
  To prove the second identity, note that for each positive divisor $d$ of $N$,
  $$
  \sum_{i=0}^{N-1} \eta_i^{(d,q^m)} = \dfrac{N}{d}\sum_{i=0}^{d-1} \eta_i^{(d,q^m)}  = \dfrac{N}{d} \sum_{x \in \Fr^*} \chi_{q^m}(x) = -\dfrac{N}{d}.
  $$
  Hence
  $$
  \sum_{i=0}^{N-1} \Delta_i(N) = \sum_{d \mid N} \mu(d) \sum_{i=0}^{N-1} \eta_i^{(d,q^m)} = -\sum_{d \mid N} \mu(d)\dfrac{N}{d} = -\phi(N).
  $$
  For the last identity, by (\ref{eqn: periods and gauss sums}) and Lemma \ref{lem: characteristic}, we have
  \begin{align*}
   \Delta_k(N) &= \sum_{d \mid N} \mu(d) \eta_k^{(d,q^m)} = \sum_{d \mid N} \dfrac{\mu(d)}{d} \sum_{j=0}^{d-1} \overline{\psi_d^j \left( \alpha^k \right)} G_{q^m} \left( \psi_d^j \right) \\
   &=
   \sum_{d \mid N}\dfrac{\mu(d)}{d} \sum_{j=0}^{d-1} \overline{\psi_d^j \left( \alpha^k \right)} \sum_{i=1}^{q^m-1} \chi_{q^m}\left(  \alpha^i \right) \psi_d^j\left( \alpha^{i}  \right)\\
   &=
   \sum_{i=1}^{q^m-1} \chi_{q^m}\left(  \alpha^i \right) \sum_{d \mid N}\dfrac{\mu(d)}{d} \sum_{j=0}^{d-1}\psi_d^j \left( \alpha^{i-k} \right)\\
   &=
   \sum_{\substack{i=1 \\ \gcd(N, i-k) = 1}}^{q^m-1} \chi_{q^m}\left(  \alpha^i \right).
  \end{align*}
\end{proof}

For the sake of brevity let us also fix the following notation for the remaining of the paper.
\begin{align*}
f_k(N, c, \Delta) &:= \sum_{i=0}^{q-2}\overline{\chi_q\left( \alpha^{Qi}c \right)} \Delta_{Qi + k}(N).
\end{align*}

Now we give the general formula for $Z_{q,m,N}(c)$.

\begin{lem}\label{lem: Z formula}
 Let $N$ be a positive divisor of $q^m-1$ and let $c$ be an arbitrary element of $\F$. Then 
 \begin{align*}
 Z_{q,m, N}(c) 
&= \dfrac{1}{q}\left( \dfrac{q^m-1}{N} \phi(N) + f_0(N, c, \Delta)  \right).
 \end{align*}
\end{lem}

\begin{proof}
By the orthogonality relation in (\ref{eqn: orthogonality}) and by the transitivity of the trace function, note that
$$
\dfrac{1}{q} \sum_{a \in \F} \overline{\chi_q(ac)} \chi_{q^m} \left( a \alpha^k \right) = \dfrac{1}{q}\sum_{a \in \F} \chi_q \left( a \left( \Tmq\left(\alpha^k  \right)  - c\right)  \right) = \begin{cases}
                                                                                       1 & \mbox{ if } \Tmq(\alpha^k) = c;\\
                                                                                       0 & \mbox{ otherwise.}
                                                                                      \end{cases}
 $$
 Thus if we multiply the characteristic function above with that of Lemma \ref{lem: characteristic}, and then sum over all the elements in $\Fr^*$, we get
 \begin{align*}
  Z_{q,m,N}(c) 
  &= \dfrac{1}{q} \sum_{d \mid N} \mu(d)\sum_{a \in \F} \overline{\chi_q(ac)} \dfrac{1}{d}\sum_{j=0}^{d-1}\sum_{k=1}^{q^m-1}\chi_{q^m} \left( a \alpha^k \right) \psi_d^j\left( \alpha^k  \right)\\
  &= \dfrac{1}{q} \left( \sum_{k=1 }^{q^m-1} \sum_{d \mid N}   \dfrac{\mu(d)}{d}\sum_{j=0}^{d-1} \psi_d^j\left(\alpha^k  \right)
  +  \sum_{i=0}^{q-2} \overline{\chi_q\left(\alpha^{Qi}c \right)} \sum_{d \mid N} \mu(d)\dfrac{1}{d}\sum_{j=0}^{d-1}\sum_{x \in \Fr^*}\chi_{q^m} \left( \alpha^{Qi} x \right) \psi_d^j\left( x  \right) \right)\\
  &= \dfrac{1}{q} \left( \dfrac{q^m-1}{N} \phi(N) + \sum_{i=0}^{q-2} \overline{\chi_q\left(\alpha^{Qi}c \right)} \sum_{d \mid N} \mu(d) \eta_{Qi}^{(d,q^m)}    \right)\\
  &= \dfrac{1}{q} \left( \dfrac{q^m-1}{N} \phi(N) + f_0(N, c, \Delta)   \right).
 \end{align*}
\end{proof}

\section{The case of the zero trace}

In this section we consider the special case of the zero trace and prove some of the corresponding assertions already mentioned in the introduction, and give some other related results.
We start off in Subsection 4.1 by deriving Theorem \ref{lem: zero trace} and giving some immediate consequences. See Corolleries \ref{thm: N coprime to Q}, \ref{thm: Q prime with trace 0},  \ref{thm: N = 2} and \ref{thm: N = 3} in the introduction. 
Then in Subsection 4.2 we prove Corollaries \ref{thm: semi delta} and \ref{thm: tuxanidy}, 
 and also  prove the ``semi-primitive'' characterization in Proposition \ref{prop: mersenne}. 

\subsection{Simplification of $ Z_{q,m,N}(0)$ and direct consequences}

First, in the case of the zero trace, we apply Ding-Yang lemmas (Lemma \ref{lem: ding} and \ref{lem: ding2}) to simplify the expression for $f(N, 0, \Delta)$.

\begin{lem}\label{lem: delta sum}
 Let $N \mid q^m-1$,  $k \in \mathbb{Z}$ and $K_Q$ be the largest divisor of $N$ that is coprime to $Q$. Then
 $$
 f_k(N,0, \Delta) = (q-1) \dfrac{\phi(K_Q)}{K_Q} \Delta_k(\gcd(Q, N)).
 $$
\end{lem}

\begin{proof}
For a positive divisor $d$ of $q^m-1$, let us denote, for the sake of brevity, $g(d) = \gcd(Q, d)$.
Now, by Lemma \ref{lem: ding2},
\begin{align*}
  f_k(N,0, \Delta) &= \sum_{d \mid N} \mu(d) \sum_{i=0}^{q-2}\eta_{Qi + k}^{(d,q^m)}\\
  &= (q-1)\sum_{d \mid N} \mu(d) \dfrac{g(d)}{d} \eta_k^{(g(d),q^m)}.
 \end{align*}
Note $\rad(N) = \rad(K_Q) \rad(\gcd\left(Q, N\right))$ is the product of the two coprime numbers $\rad(K_Q)$ and $\rad(\gcd\left(Q, N\right))$. 
 Then we can write any positive divisor $d$ of $\rad(N)$ uniquely as $d = yz$, where $y \mid \rad(K_Q)$ and $z \mid \rad(\gcd\left(Q, N\right))$. Moreover $g(yz) = z$ for any such $y,z$. Hence 
 \begin{align*}
  f_k(N,0, \Delta) &= (q-1) \sum_{y \mid K_Q} \sum_{z \mid \gcd\left(Q, N\right)} \mu(yz) \dfrac{g(yz)}{yz} \eta_k^{(g(yz),q^m)}\\
  &= (q-1) \sum_{y \mid K_Q} \dfrac{\mu(y)}{y} \sum_{z \mid \gcd\left(Q, N\right)} \mu(z) \eta_k^{(z,q^m)}\\
  &= (q-1) \dfrac{\phi(K_Q)}{K_Q} \Delta_k(\gcd\left(Q, N\right)).
 \end{align*}
\end{proof}

\begin{proof}[{\bf Proof of Theorem \ref{lem: zero trace}}]
 By Euler's product formula for $\phi$, and using the fact that $K_Q$ is coprime to $\gcd\left(Q, N\right)$ with $\rad(N) =  \rad(K_Q)\rad(\gcd\left(Q, N\right))$, we have
 \begin{align*}
  \dfrac{\phi(N)}{N} &= \prod_{\ell \mid N} \left( 1 - \dfrac{1}{\ell}  \right) \\
  &= \left(\prod_{\ell \mid K_Q } \left( 1 - \dfrac{1}{\ell}  \right)\right) \left(\prod_{\ell \mid \gcd\left(Q, N\right) }\left( 1 - \dfrac{1}{\ell}  \right)   \right)\\
  &= \dfrac{\phi(K_Q)}{K_Q} \dfrac{\phi(\gcd\left(Q, N\right))}{\gcd\left(Q, N\right)},
  \end{align*}
  where in the three products $\ell$ runs through all the distinct prime divisors of $N, K_Q$ and $\gcd\left(Q, N\right)$, respectively. 
  Hence by Lemmas \ref{lem: Z formula} and \ref{lem: delta sum},
  \begin{align*}
    Z_{q,m,N}(0) &= \dfrac{1}{q} \left( (q^m-1) \dfrac{\phi(N)}{N} + (q-1) \dfrac{\phi(K_Q)}{K_Q} \Delta_0(\gcd\left(Q, N\right))  \right)\\
   &= \dfrac{\phi(K_Q)}{qK_Q} \left( (q^m-1) \dfrac{\phi(\gcd\left(Q, N\right))}{\gcd\left(Q, N\right)} + (q-1)  \Delta_0(\gcd\left(Q, N\right))  \right)\\
   &= \dfrac{(q-1)\phi(K_Q)}{qK_Q}\left( \dfrac{Q}{\gcd\left(Q, N\right)}\phi(\gcd\left(Q, N\right)) +  \Delta_0(\gcd\left(Q, N\right))  \right).
  \end{align*}
\end{proof}

In particular one obtains in Lemma \ref{lem: primitive zero trace} the number of primitives with zero trace. The second equality (on the right) gives Lemma 2.1 in \cite{cohen and presern (2005)}.

\begin{lem}\label{lem: primitive zero trace}
Let $D$ be the smallest positive divisor of $q-1$ such that $(q-1)/D$ is coprime to $Q$. Then the number of primitive elements $\xi$ in $\Fr$ with $\Tmq(\xi) = 0$ is given by
 \begin{align*}
  Z_{q,m,q^m-1}(0) &= D \phi\left( \dfrac{q-1}{D} \right)\dfrac{ \phi(Q) + \Delta_0(Q)  }{q} \\
  &= D \phi\left( \frac{q-1}{D}  \right) \dfrac{Z_{q,m,Q}(0)}{q-1} .
 \end{align*}
\end{lem}

 We now derive some other immediate consequences to Theorem \ref{lem: zero trace}.  
 
 \begin{proof}[{\bf Proof of Corollary \ref{thm: N coprime to Q}}]
  Follows directly from Theorem \ref{lem: zero trace} and the fact that $\Delta_0(1) = -1$.
 \end{proof}
 
 \begin{lem}\label{lem: delta Q prime}
 Let $q$ be a power of a prime $p$ and assume that $Q = (q^{\ell}-1)/(q-1)$ is prime for some prime $\ell$. Then
 $$
 \Delta_0(Q) = \begin{cases}
                0 & \mbox{ if } \ell \neq p;\\
                -q & \mbox{ otherwise.}
               \end{cases}
$$
\end{lem}

\begin{proof}
 Since $Q$ is prime, then $\Delta_0(Q) = -1 - \eta_0^{(Q, q^{\ell})}$. Now the result follows from Lemma \ref{lem: period Q}.
\end{proof}

\begin{proof}[{\bf Proof of Corollary \ref{thm: Q prime with trace 0}}]
 Note that $\gcd(q-1, Q) = 1$ since otherwise $Q \mid q-1$ contradicting $Q > q-1$ for $\ell \geq 2$. 
  Now the result follows directly from Lemmas \ref{lem: delta Q prime} and \ref{lem: primitive zero trace}.
 \end{proof}

%\begin{thm}\label{thm: N = 2}
%Let $q = p^s$, let $r = q^m$, let $Q = (q^m-1)/(q-1)$, let $N \mid q^m-1$ such that $\gcd(Q, N) = 2^n$ for %some $n \geq 1$, and let $K_Q$ be the largest odd divisor of $N$. Then
 %$$
 %\dfrac{2 q K_Q }{(q-1)\phi(K_Q)}  Z_{q,m,N}(0) = \begin{cases}
  %                 Q - 1 + (-1)^{sm}q^{m/2}   & \mbox{ if } p \equiv 1 \pmod{4};\\
   %               Q - 1 +  \left(- \sqrt{-1} \right)^{sm}q^{m/2}   & \mbox{ if } p \equiv 3 \pmod{4}.
    %             \end{cases}
%$$
%\end{thm}

 \begin{proof}[{\bf Proof of Corollary \ref{thm: N = 2}}]
  Follows directly from Theorem \ref{lem: zero trace} and  Lemma \ref{lem: period N = 2}.
 \end{proof}
 
% \begin{thm}\label{thm: N = 3}
 %  Let $p \equiv 1 \pmod{3}$, 
  % let $q = p^s$, let $r = q^m$, let $Q = (q^m-1)/(q-1)$, let $N \mid q^m-1$ such that $\gcd(N,Q) = 3^n$ for some $n \geq 1$, and let $K_Q$ be the 
  % largest divisor of $N$ with $3 \nmid K_Q$.
  % Let $c \in \mathbb{Z}$ with $c \equiv 1 \pmod{3}$ and $p \nmid c$ be a solution to the equation
  % $4q^{m/3} = c^2 + 27d^2$ with $d \in \mathbb{Z}$.  
  % Then
  % $$
   % Z_{q,m,N}(0) = \dfrac{(q-1)\phi(K_Q)}{qK_Q} \left( \dfrac{2Q - 2 - cq^{m/3}}{3}  \right).
   %$$
   %\end{thm}
   
   \begin{proof}[{\bf Proof of Corollary \ref{thm: N = 3}}]
    Note that since $\gcd(Q, N) = 3$, then $Q \equiv m \equiv 0 \pmod{3}$. The result now follows from 
    Lemma \ref{lem: period N = 3} and Theorem \ref{lem: zero trace}.
   \end{proof}
   
   \iffalse
   As a result of Theorem \ref{thm: cohen} one obtains the following  inequality for the sum $\Delta_0(Q)$.
As we shall see later on, for any subfield $\mathbb{F}_t$ of $\F$ and every non-zero element $c_t \in \mathbb{F}_t^*$, 
we have $\Delta_0(Q) = Z_{t,m[\mathbb{F}_q:\mathbb{F}_t], Q}(0) - Z_{t,m[\mathbb{F}_q:\mathbb{F}_t], Q}(c_t \neq 0)$, where, as before, $Q = (q^m-1)/(q-1)$.

\begin{prop}
 Let $q$ be a power of a prime and  $Q = (q^m-1)/(q-1)$. Then we have the following inequality:
 $$
 \Delta_0(Q) + \phi(Q) \geq 0 
 $$
 with equality holding if and only if $m = 1, 2,$ or $(q,m) = (4,3)$.
 \end{prop}

 \begin{proof}
  This is a direct consequence of Theorem \ref{thm: cohen} and the fact that $\Delta_0(1) = -1$ for the case of $m=1$.
 \end{proof}
 \fi

\subsection{Proof of Corollaries \ref{thm: semi delta} and \ref{thm: tuxanidy}} 

In this subsection we prove Corollaries \ref{thm: semi delta} and \ref{thm: tuxanidy} stated in the introduction, corresponding to the case of the zero trace. 
We employ the known explicit formulas for the Gaussian periods 
in the semi-primitive case (see Lemma \ref{lem: semiprimitive}) to first derive, in the following lemma, 
the value of the sum $\Delta_0(N)$ for $N$ falling under the category of Lemma \ref{lem: semiprimitive}. 
Then by Theorem \ref{lem: zero trace} we get the result of Corollary  \ref{thm: semi delta}.
One of course can then naturally consider whether this result applies to primitives, but unfortunately, as Proposition \ref{prop: mersenne} shows, 
it only extends to primitives in quartic extensions of Mersenne prime fields. 
%However given the apparent absence of results in literature regarding the number of primitives with %prescribed trace, 
%then, in the opinion of the authors, the result of Corollary \ref{lem: zero trace} stands on its own. Interestingly, 
Mersenne primes also appear in the trivial case for which a formula is known. This is the case when all the irreducibles are also the primitives, that is, when
$q = 2$ and $m = \ell$ with $2^\ell -1$ being a Mersenne prime. See the comments under Theorem \ref{thm: carlitz} in the introduction.

\begin{lem}\label{lem: semi delta}
 Let $sm$ be even with $m > 1$, let $q = p^s$ be a power of a prime $p$, and suppose that $n > 1$ is not a power of $2$ and satisfies $n \mid q^m-1$.
    Further assume there exists a positive integer $j$ such that $p^j \equiv -1 \pmod{\ell}$ for every prime divisor $\ell \geq 3$ of $n$,
    and that $j$ is the least such. Define $\gamma = sm/2j$ and let $\eta_0^{(2,q^m)}$ be as in Lemma \ref{lem: period N = 2}.
    
    (a) If $\gamma$ and $p$ are odd, $n$ is even and $(p^j+1)/2$ is odd, then
    $$
    \Delta_0(n) = - \eta_0^{(2,q^m)} - \left( 1 + q^{m/2}   \right)\left( \dfrac{1}{2} + \dfrac{ \phi\left( n \right) }{n}  \right).
    $$
    
    (b) In all other cases,
    $$
    \Delta_0(n) = \epsilon_2 \cdot \left(\dfrac{(-1)^{\gamma + 1}q^{m/2} - 1}{2} - \eta_0^{(2,q^m)}\right) + \dfrac{(-1)^{\gamma} q^{m/2} -1}{n} \phi(n),
    $$
    where
    $$
    \epsilon_2 = \begin{cases}
                  1 & \mbox{ if } n \mbox{ is even;}\\
                  0 & \mbox{ otherwise.} 
                 \end{cases}
$$
\end{lem}

\begin{proof}
    First note the assumption on $n$ and $j$ means that, for every positive $3 \leq d \mid \rad(n)$, we have that $j$ is the least such that $p^j \equiv -1 \pmod{d}$. 
    
    (a) Consider any such $d$ as above. If $d$ is odd, then $(p^j+1)/d$ is even and so $\eta_0^{(d,q^m)}$ belongs to case
    (b) of Lemma \ref{lem: semiprimitive}. Moreover since $2$ has multiplicity $1$
    in the factorization of $p^j+1$, then $(p^j+1)/2d$ is odd if so is $d$; in this case
    $\eta_0^{(2d, q^m)}$ belongs to case (a) in Lemma \ref{lem: semiprimitive}. Now let $V_2$ be the largest power of $2$ dividing $n$. We then have
    \begin{align*}
  \sum_{3 \leq d \mid n } \mu(d) \eta_0^{(d,q^m)} 
  &=  \sum_{\substack{3 \leq d \mid n \\ d \text{ odd}}} \left( \mu(d) \eta_0^{(d,q^m)} + \mu(2d) \eta_0^{(2d, q^m)} \right)
  =  \sum_{\substack{3 \leq d \mid n \\ d \text{ odd}}} \mu(d) \left( \eta_0^{(d,q^m)} - \eta_0^{(2d, q^m)}   \right)\\
  &= \sum_{\substack{3 \leq d \mid n \\ d \text{ odd}}} \mu(d) \left( \dfrac{(d-1)q^{m/2} - 1}{d}  + \dfrac{q^{m/2}+1}{2d}  \right)
  = \sum_{\substack{3 \leq d \mid n \\ d \text{ odd}}} \mu(d)\left( q^{m/2} - \dfrac{q^{m/2}+ 1}{2d}    \right)\\
  &= \dfrac{1-q^{m/2}}{2} + \sum_{\substack{d \mid n \\ d \text{ odd}}} \mu(d)\left( q^{m/2} - \dfrac{q^{m/2}+ 1}{2d}    \right)
  = \dfrac{1-q^{m/2}}{2} + \sum_{d \mid n/V_2} \mu(d)\left( q^{m/2} - \dfrac{q^{m/2}+ 1}{2d}    \right)\\
  &= \dfrac{1-q^{m/2}}{2} - \dfrac{q^{m/2}+ 1}{2}\sum_{d \mid n/V_2} \dfrac{\mu(d)}{d}
  = \dfrac{1 - q^{m/2}}{2} - \dfrac{(1 + q^{m/2})V_2}{2n} \phi\left( \dfrac{n}{V_2} \right)\\
  &= \dfrac{1 - q^{m/2}}{2} - \dfrac{1 + q^{m/2}}{n} \phi\left( n \right)
 \end{align*}
 since $V_2$ is coprime to $n/V_2$ and $\phi(V_2) = V_2/2$. Then we have
 $$
 \Delta_0(n) = -1 - \eta_0^{(2,q^m)} + \sum_{3 \leq d \mid n} \mu(d) \eta_0^{(d,q^m)} = -1  - \eta_0^{(2,q^m)} + \dfrac{1 - q^{m/2}}{2} - \dfrac{1 + q^{m/2}}{n} \phi\left( n \right).
 $$
Hence the result follows. 
 
 (b) As before assume $d \mid n $ with $d \geq 3$ and $\mu(d) \neq 0$. We claim that 
 $\eta_0^{(d,q^m)}$ belongs to case (b) in Lemma \ref{lem: semiprimitive}. Indeed, 
if $p$ or $\gamma$ is even, then $\eta_0^{(d,q^m)}$ belongs to $(b)$. If $d$ is odd, necessarily $(p^j + 1)/d$ is even unless $p$ is even; hence (b). This also takes care of the case
when $n$ is odd. Finally if $n$ is even and $2$ has multiplicity greater than $1$ in the factorization of $p^j+1$, then $(p^j+1)/2d$ is even for any such $d$ odd.
   The claim follows. Hence by Lemma \ref{lem: semiprimitive} (b) we have 
   \begin{align*}
    \sum_{3 \leq d \mid n} \mu(d) \eta_0^{(d, r)} &= \sum_{3 \leq d \mid n} \mu(d)  \left(\dfrac{(-1)^{\gamma + 1}dq^{m/2} + (-1)^{\gamma}q^{m/2} - 1}{d}   \right)\\
&= 1 + \epsilon_2 \cdot \left(\dfrac{(-1)^{\gamma + 1}q^{m/2} - 1}{2}\right) + \sum_{d \mid n}\mu(d) \left( (-1)^{\gamma +1}q^{m/2} + \dfrac{(-1)^{\gamma} q^{m/2} -1}{d} \right)\\
&= 1 + \epsilon_2 \cdot \left(\dfrac{(-1)^{\gamma + 1}q^{m/2} - 1}{2}\right) + \sum_{d \mid n}\mu(d)\left(\dfrac{(-1)^{\gamma} q^{m/2} -1}{d}\right)\\
&= 1 + \epsilon_2 \cdot \left(\dfrac{(-1)^{\gamma + 1}q^{m/2} - 1}{2}\right) + \dfrac{(-1)^{\gamma} q^{m/2} -1}{n}\sum_{d \mid n} \mu(d)\dfrac{n}{d}\\
&= 1 + \epsilon_2 \cdot \left(\dfrac{(-1)^{\gamma + 1}q^{m/2} - 1}{2}\right) + \dfrac{(-1)^{\gamma} q^{m/2} -1}{n} \phi(n).
   \end{align*}
Hence the result follows.
   \end{proof}

\begin{proof}[{\bf Proof of Corollary  \ref{thm: semi delta}}]
 It follows directly from Lemma \ref{lem: semi delta} and Theorem \ref{lem: zero trace}.
\end{proof}

In order to prove Proposition \ref{prop: mersenne} we make use of Mih\u{a}ilescu's breakthrough result \cite{catalan}, also known as Catalan's conjecture. 
Although like Wiles' Theorem (Fermat's Last Theorem) it is easily stated, it took 160 years for the conjecture to be finally solved, by Mih\u{a}ilescu \cite{catalan}.
   
\begin{thm}[{\bf Mih\u{a}ilescu (2004)}]\label{thm: catalan}
 Let $x,y,a,b \in \mathbb{N}$ with $a, b > 1$. If $x^a - y^b = 1$, then $x = b = 3$ and $y = a = 2$.
\end{thm}

For a positive integer $k$, we let $v_2(k)$ denote the multiplicity of $2$ in the factorization of $k$. For an integer $b$ coprime to $k$, 
we let $\ord_k(b)$ denote the multiplicative order of $b$ modulo $k$.

\begin{proof}[{\bf Proof of Proposition \ref{prop: mersenne}}]
 Let $q = p^s$ be the power of a prime $p$, with $s \geq 1$. The case when $m = 2$ is clear as $(q^2 - 1)/(q-1) = q+1$ and so we can let $j = s$. We may thus suppose $m > 2$. 
 Now assume that the congruence is satisfied, i.e., that
 $$
 p^j \equiv -1 \pmod{\rad\left(\dfrac{p^{sm}-1}{p^s-1}\right)}
 $$
 holds. 
 
 We claim that $m$ is a power of $2$. On the contrary, suppose the odd part, $k$, of $m$, satisfies $k > 1$. 
 Let $t$ be the odd part of $sk$. Then $t > 1$ and $(sk, t) = t\nmid s$ since $k > 1$ and $k \mid t$. Since $t \mid sk$, we have
 $(p^{sk} - 1, p^t - 1) = p^t - 1 > 1$. 
 Moreover $d := \rad((\frac{p^{sk} - 1}{p^s - 1}, p^t - 1)) > 1$ since otherwise $(p^t - 1) \mid (p^s - 1)$ contrary to $t \nmid s$. 
 Because $k > 1$ is odd, then both $\frac{p^{sk} - 1}{p^s - 1}$ and hence $d$ are odd. In particular $d \geq 3$.
 
 Note that $p^t \equiv 1 \pmod{d}$. Thus if $a$ is the smallest positive 
 integer such that $p^a \equiv 1 \pmod{d}$, then $a \mid t$ and hence $a$ is odd. 
 Now observe that $p^j \equiv -1 \pmod{d}$, since $d \mid \rad(\frac{p^{sk} - 1}{p^s - 1}) \mid \rad(\frac{p^{sm} - 1}{p^s - 1})$. Thus $p^{2j} \equiv 1 \pmod{d}$. 
 It follows that $a \nmid j$ but $a \mid 2j$. This however implies that $a$ is even, a contradiction. The claim follows. Hence $m = 2^n$ for some $n \geq 2$ and
 $$
 \dfrac{p^{sm} - 1}{p^s-1} = \prod_{i=0}^{n-1} \left( p^{s2^{i}} + 1  \right).
 $$
 
 Since $n \geq 2$, we can let $i \geq 0$ be arbitrary such that both $p^{s2^i} + 1$ and $p^{s2^{i+1}} + 1$ divide $(p^{s2^n} - 1)/(p^s-1)$. 
 Suppose there exist square-free divisors $d_i, d_{i+1} \geq 3$ of $p^{s2^i} + 1$ and $p^{s2^{i+1}} + 1$, respectively. 
 For the sake of brevity let $a_i := \ord_{d_i}(p)$ and $a_{i+1} := \ord_{d_{i+1}}(p)$. Since $p^j \equiv -1 \pmod{d_i}$, we have $a_i \nmid j$ but $a_{i} \mid 2j$ 
 implying that $v_2(j) = v_2(a_i) - 1$. Similarly since $p^{s2^i} \equiv -1 \pmod{d_i}$, then $a_i \nmid s2^i$ but $a_i \mid 2(s2^i)$. This implies that $v_2(a_i) = v_2(s) + i + 1$. It then follows that
 $v_2(j) = v_2(s) + i$. Working with $d_{i+1}$ and $a_{i+1}$ now we similarly deduce that $v_2(j) = v_2(s) + i + 1$, a contradiction. Necessarily at least one of
 $p^{s2^i} + 1$, $p^{s2^{i+1}} + 1$, is a power of $2$. But note that, for any positive integer $b$, 
 we have $p^b + 1 = 2^v$ if and only if $p^b = 2^v - 1$. In this case it is clear that $v > 1$; then
 we can not have
 $b > 1$ since otherwise Theorem \ref{thm: catalan} is contradicted. In particular $p$ is a Mersenne prime. Moreover 
 $s2^i = 1$ or $s2^{i+1} = 1$. Of these two, only $s2^i = 1$ is possible, whence $i = 0$, $s = 1$, and $q = p$ is a Mersenne prime. 
 Because $i \geq 0$ was arbitrary such that both $p^{s2^i} + 1$ and $p^{s2^{i+1}} + 1$ divide $(p^{s2^n} - 1)/(p^s-1)$, necessarily $n = 2$ is the only possibility, whence
 $m = 4$ and $(p^{sm} - 1)/(p^s - 1) = (p+1)(p^2 + 1)$. 
  
 Since $p$ is a Mersenne prime, $p + 1$ is a power of $2$ and $2 \mid (p^2 + 1)$ imply that $\rad((p+1)(p^2 + 1)) = \rad(p^2 + 1)$. 
 Thus it only remains to show that $\rad(p^2 + 1) \mid (p^j + 1)$ if and only if $j = 2k$ for some odd $k \geq 1$. 
 Assume that $\rad(p^2 + 1) \mid (p^j + 1)$. First it is easy to check that $p^2 + 1$ is not a power of $2$. 
 Hence we can let $d \geq 3$ with $d \mid \rad(p^2 + 1)$. Then $p^2 \equiv -1 \pmod{d}$ and $p^4 \equiv 1 \pmod{d}$ implies $\ord_d(p) = 4$. 
Since $d \mid \rad(p^2 + 1) \mid (p^j + 1)$, then $p^j \equiv -1 \pmod{d}$ implying $4 \nmid j$ but $4 \mid 2j$. This means that $j = 2k$ for some odd $k \geq 1$.
Conversely let $j = 2k$ for any odd $k \geq 1$. Note that $p^{2k} = (p^2)^k = [(p^2 + 1) - 1]^k$. Then by the Binomial Theorem we have
\begin{align*}
p^{2k} &= \left[\left(p^2 + 1\right) - 1\right]^k = \sum_{i=0}^k {k \choose i} \left( p^2 + 1 \right)^i (-1)^{k-i} 
= -1 + \sum_{i=1}^k {k \choose i} \left( p^2 + 1 \right)^i (-1)^{k-i}.
\end{align*}
Hence
$$
p^{2k} + 1 = \sum_{i=1}^k {k \choose i} \left( p^2 + 1 \right)^i (-1)^{k-i}
$$
and so $(p^2 + 1) \mid (p^{2k} + 1)$. Consequently $\rad(p^2 + 1) \mid (p^{2k} + 1)$. 
\end{proof}

\begin{proof}[{\bf Proof of Corollary \ref{thm: tuxanidy} }]
By Proposition \ref{prop: mersenne} we have that $Q = (p^4 - 1)/(p-1) = (p+1)(p^2 + 1)$ satisfies $p^j \equiv -1 \pmod{\rad(Q)}$ with $j = 2k$ for every odd $k \geq 1$; 
in particular for $j = 2$. If $\ell \geq 3$ is a prime divisor of $Q$, clearly $\ell \nmid p + 1$ (since $p+1$ is a power of $2$). Then $j = 2$ is the least such that
$p^j \equiv -1 \pmod{\ell}$. 
Corollary \ref{thm: semi delta} then applies with $s = 1$, $m = 4$,  $N = p^4-1$, $j = 2$, $n = Q = (p+1)(p^2+1)$, and $\gamma = 1$. With the notation of Corollary  \ref{thm: semi delta}, 
it is easy to show that $K_Q = (p-1)/2$. Since $p, \gamma,$ are odd while $n$ is even and $2$ has multiplicity $1$ in the factorization of $p^2+1$, then case (a) of Lemma \ref{lem: semi delta}
applies. In this case we get, using the fact that $p \equiv 3 \pmod{4}$ 
together with Lemma \ref{lem: period N = 2},
\begin{align*}
 \Delta_0(p^2 + 1) &= \Delta_0((p+1)(p^2+1)) = -1 + \dfrac{1 + p^2}{2} + \dfrac{1 - p^2}{2} - \dfrac{1}{p+1} \phi\left( (p+1)(p^2+1)  \right)\\
 &= - \dfrac{1}{p+1} \phi((p+1)(p^2+1)) = - \phi(p^2+1) = -\phi\left(\dfrac{p^2+1}{2}\right). 
\end{align*}
It then follows from Lemma \ref{lem: primitive zero trace} that
 $$
 Z_{p,4,p^4-1}(0) = \dfrac{1}{p} \left( \phi\left(p^4-1 \right) - 2 \phi\left( \dfrac{p-1}{2}  \right)   \phi\left( \dfrac{p^2+1}{2}   \right)\right).
 $$
 It remains to note that 
 $$\phi\left(\dfrac{p-1}{2} \right) \phi\left(\dfrac{p^2+1}{2}\right) = \phi\left(\dfrac{(p-1)(p^2+1)}{4}\right) = \dfrac{1}{2}\phi\left(\dfrac{p^4-1}{p+1}\right)$$
 since 
 $$\gcd\left(\dfrac{p-1}{2}, \dfrac{p^2+1}{2}\right) = \gcd\left(4, \dfrac{(p-1)(p^2+1)}{4}\right) = 1$$
 and
 $p^4-1 = (p-1)(p+1)(p^2+1)$.
 The result follows.
\end{proof}

\section{Uniformity in the case of the non-zero trace}

In this section we explore the concept of uniformity, already discussed in the introduction. That is, the main concern here is as follows: what triples $(q,m,N)$, with $N \mid q^m-1$, 
are such that $Z_{q, m,N}(c)$ is constant for every non-zero $c \in \F^*$? Accordingly, in this section we prove Theorem \ref{thm: uniformity}. As a consequence of this and of Corollary  \ref{thm: semi delta}, 
Corollary \ref{cor: semidelta uniform} is straightforward. In particular, in the case of primitives, i.e., $N = q^m-1$, we give sufficient conditions for $Z_{q,m,q^m-1}(c)$ to behave uniformly for $c \in \F^*$. 
See Corollary \ref{cor: unif} for this.

\begin{lem}\label{lem: uniformity}
 Let $N \mid q^m-1$,  $c \in \F^*$ be arbitrary, $K_Q$ be the largest divisor of $N$ that is coprime to $Q$. Then
 $$
 Z_{q,m,N}(c \neq 0) = \dfrac{1}{q} \left(  \dfrac{q^m-1}{N} \phi(N) + K_Q \dfrac{ \frac{q^m-1}{N}\phi(N) - q Z_{q,m,N}(0)    }{(q-1) \phi(K_Q)  } + f_0(N,c, \Delta) - f_0\left(\gcd\left(Q, N\right), c, \Delta\right)\right).
 $$
 \end{lem}
 
 \begin{proof}
  By Proposition \ref{prop: identities} and using the fact that $\sum_{a \in \F^*} \overline{\chi_q(ac)} = -1 $ for $c \neq 0$, we get
  \begin{align*}
   f_0(\gcd\left(Q, N\right), c, \Delta) &= \sum_{i=0 }^{q-2} \overline{\chi_q\left( \alpha^{Qi} c \right)}  \Delta_{Qi}(\gcd\left(Q, N\right)) = \sum_{i=0 }^{q-2} \overline{\chi_q\left( \alpha^{Qi} c \right)}\Delta_{0}(\gcd\left(Q, N\right))\\
   &= -\Delta_0(\gcd\left(Q, N\right)).
  \end{align*}
  Hence
  $$
  f_0(N, c, \Delta) = - \Delta_0(\gcd\left(Q, N\right)) + f_0(N, c, \Delta ) - f_0(\gcd\left(Q, N\right), c, \Delta).
  $$
  By Theorem \ref{lem: zero trace} we can write $\Delta_0(\gcd\left(Q, N\right))$ in terms of $ Z_{q,m,N}(0)$. Now the expression for $Z_{q,m,N}(c \neq 0)$ follows from Lemma \ref{lem: Z formula}.
 \end{proof}

 \begin{proof}[{\bf Proof of Theorem \ref{thm: uniformity}}]
  If every prime divisor of $N$ divides $Q$, then $\rad(N) = \rad(\gcd\left(Q, N\right))$ and $K_Q = 1$. Now the result follows from Lemma \ref{lem: uniformity}
  together with the fact that $f_0(d, c, \Delta) = f_0(\rad(d), c, \Delta)$ for every $d \mid q^m-1$.
 \end{proof}
 
   \begin{lem}\label{lem: legendre 0}
  Let $b, m > 1 $. Then $\rad(b^m-1) = \rad(\frac{b^m-1}{b-1})$ if and only if every prime factor of $b-1$ divides $m$. 
 \end{lem}
 
 \begin{proof}
 First note that $\rad(b^m-1) = \rad(\frac{b^m-1}{b-1})$ if and only if $\rad(b-1) \mid \rad(\frac{b^m-1}{b-1})$.
 Now
 \begin{align*}
  \dfrac{b^m-1}{b-1} &= 1 + b + b^2 + \cdots  + b^{m-1}\\
  &= (1 - b) + (b - 1) + (b^2 - 1) + \cdots + (b^{m-1} - 1) \\
  & \quad + b + m - 1.
 \end{align*}
It follows that $\rad(b-1) \mid \rad(\frac{b^m-1}{b-1})$ if and only if $\rad(b-1)$ divides $m$. 
 \end{proof} 

As a consequence of Lemma \ref{lem: legendre 0}, we obtain the following immediate result. 

 \begin{cor}\label{cor: unif}
   Let $q$ be a power of a prime and $m \in \mathbb{N}$. If every prime factor of $q-1$ divides $m$, 
  then, for every element $c \in \F^*$, the number, $Z_{q,m,q^m-1}(c)$, of primitive elements $\xi \in \Fr$ with $\Tmq(\xi) = c$, is given by
  $$
  Z_{q,m,q^m-1}(c \neq 0) = \dfrac{\phi(q^m-1) -  Z_{q,m,q^m-1}(0)}{q-1}.
  $$
 \end{cor}

 Some other consequences to Theorem \ref{thm: uniformity} are the following.

 \begin{cor}\label{cor: uniformity 0}
  Let $N \mid q^m-1$ such that $\rad(N) \mid Q$ and let $\mathbb{F}_{t}$ be any subfield of $\F$. Then, for all $c_t \in \mathbb{F}_{t}^*$ and $c_q \in \F^*$, we have
  $$
  Z_{t, m[\F:\mathbb{F}_t], N}(c_t)  = \dfrac{\frac{q^m-1}{N} \phi(N) - Z_{t, m[\F:\mathbb{F}_t], N}(0) }{t-1} = \dfrac{q}{t} Z_{q,m,N}(c_q).
  $$
 \end{cor}
 
 \begin{proof}
  Follows from Theorem \ref{thm: uniformity} together with the fact that, since $N \mid (q^m-1)/(q-1)$ and $(q^m-1)/(q-1) \mid (q^m-1)/(t-1)$, then $N \mid (q^m-1)/(t-1)$. 
 \end{proof}
 
 \begin{proof}[{\bf Proof of Corollary \ref{cor: uniformity 1}}]
  The first equality follows directly from Theorem \ref{lem: zero trace} and Theorem \ref{thm: uniformity} together with the fact mentioned in the proof of Corollary \ref{cor: uniformity 0}.
  The second equality follows from the first together with Corollary \ref{cor: uniformity 0}. Indeed, 
  \begin{align*}
   q Z_{q,m,N}(0) - t Z_{t, m[\F:\mathbb{F}_t], N}(0)  &= q \left( \Delta_0(N) + Z_{q,m,N}(c_q)  \right) - t \left( \Delta_0(N) + Z_{t, m[\F:\mathbb{F}_t], N}(0)   \right)\\
   &= (q-t) \Delta_0(N) + qZ_{q,m,N}(c_q) - tZ_{t, m[\F:\mathbb{F}_t], N}(0)  \\
   &= (q-t) \Delta_0(N).
  \end{align*}
 \end{proof}

\begin{cor}
Let $N \mid q^m-1$ such that $\rad(N) \mid Q$. Then 
  $Z_{t, m[\F:\mathbb{F}_t], N}(0)$ is related to $Z_{q,m,N}(0)$ by the equation
  $$
  (q-1) Z_{t, m[\F:\mathbb{F}_t], N}(0) = \left( q - \dfrac{q}{t}   \right) Z_{q,m,N}(0) + \left( \dfrac{q}{t} - 1   \right) \dfrac{q^m-1}{N} \phi(N).
  $$
 \end{cor}
 
 \begin{proof}
  By Corollary \ref{cor: uniformity 1} we have
  $$
  Z_{t, m[\F:\mathbb{F}_t], N}(0) - Z_{t, m[\F:\mathbb{F}_t], N}(c_t) = Z_{q,m,N}(0) - Z_{q,m,N}(c_q). 
 $$
 Hence, by Corollary \ref{cor: uniformity 0}, 
 \begin{align*}
  Z_{t, m[\F:\mathbb{F}_t], N}(0) &= Z_{q,m,N}(0)+ \left( \dfrac{q}{t} -1 \right)Z_{q,m,N}(c_q)\\
  &= Z_{q,m,N}(0)) + \left( \dfrac{q}{t} -1 \right) \left( \dfrac{ \frac{q^m-1}{N}\phi(N) - Z_{q,m,N}(0)  }{q-1}   \right)\\
  &= \dfrac{ \left(  q - \frac{q}{t}  \right) Z_{q,m,N}(0)  + \left( \frac{q}{t} - 1   \right) \frac{q^m-1}{N} \phi(N) }{q-1}.
 \end{align*}
\end{proof}

\iffalse
Since $\Delta_0(q+1) = -\phi(q+1)$ when $r = q^2$ (i.e., $m = 2$) and $Z_{q,q^2}(q^2-1, 0) = Z_{q,q^2}(q+1, 0) = 0$, the following is straightforward. 

\begin{cor}
 Let $\mathbb{F}_t$ be a subfield of $\F$. Then for every $c_t \in \mathbb{F}_t^*$, we have
 \begin{align*}
  Z_{t,q^2}(q+1, 0) &= \left( \dfrac{q}{t} - 1  \right) \phi(q+1);\\
  Z_{t,q^2}(q+1, c_t) &= \dfrac{q}{t} \phi(q+1).
 \end{align*}
\end{cor}
\fi

\section{Connection to $P_{q,m,N}(c)$}

In this section we briefly explore the seemingly related number, $P_{q,m,N}(c)$, of elements in $\Fr^*$ with order $N$ and with prescribed trace $c$. 
We start off by deriving a formula for the number of elements with order $N$ in an arbitrary subset $A$ of $\Fr$ and apply this to obtain a formula for $P_{q,m,N}(c)$ (see Lemma \ref{lem: order formula}).

Let $L_Q$ be the largest divisor of $q^m-1$ with the same radical as that of $Q$. 
In the special case of the zero trace and the case when $L_Q \mid N$, we show in Lemma \ref{lem: relation} the identity $ Z_{q,m,N}(0) = \frac{q^m-1}{N}P_{q,m,N}(0)$. 
As a consequence of this and of Cohen's result (see Theorem \ref{thm: cohen} in the Introduction) we characterize the existence of elements of order $N$ (with $L_Q \mid N$) in $\mathbb{F}_{q^m}^*$ 
with trace $0$.
Moreover from this and Corollary \ref{thm: tuxanidy} we give in Corollary \ref{thm: order DQ and mersenne} the number of elements of order $2(p+1)(p^2+1)$ with absolute trace zero in quartic extensions of 
Mersenne prime fields $\mathbb{F}_p$. 

For a subset $A \subseteq \Fr$ and a divisor $N$ of $q^m-1$, denote with $M_{q, m, N}(A)$ the number of non-zero elements in $ A$ having multiplicative order $N$ in $\Fr^*$.
In particular, $M_{q,m,q^m-1}(A)$ denotes the number of primitive elements of $\Fr$ that are contained in $A$.

\begin{lem}\label{lem: order formula}
 Let $q$ be a prime power and $m$ be a positive integer. Let $A$ be a subset of $\Fr$ and $N$ be a positive divisor of $q^m-1$. Then the number of elements 
 in $A$ that have multiplicative order $N$ is given by
 $$
 M_{q,m,N}(A) = \dfrac{1}{q^m-1}\sum_{d \mid N} \mu(d)\dfrac{N}{d} \left| \left\{ x \in \Fr^* \ : \ x^{\frac{(q^m-1)}{N}d} \in A  \right\}  \right|.
 $$
 In particular, for $c \in \F$, the number of elements $\beta \in \Fr^*$ with order $N$ and satisfying $\Tmq(\beta) = c$, is given by
 \begin{align*}
 P_{q,m,N}(c) &= \dfrac{N}{q(q^m-1)}\sum_{d \mid N} \dfrac{\mu(d)}{d}\sum_{a \in \F}\overline{\chi_q(ac)} \sum_{x \in \Fr^*} \chi_{q^m} \left( a x^{\frac{q^m-1}{N}d}  \right)\\
 &= \dfrac{1}{q} \left( \phi(N) + \sum_{i=0}^{q-2} \overline{\chi_q\left( \alpha^{Qi}c  \right)} \sum_{d \mid N} \mu(d) \eta_{Qi}^{\left( \frac{q^m-1}{N}d, r  \right)}   \right).
\end{align*}
\end{lem}

\begin{proof}
 If we define the arithmetic function $f(n) := \sum_{d \mid n}M_{q,m,d}(A),$ then by the M\"{o}bius inversion formula,
 $$
 M_{q,m,N}(A) = \sum_{d \mid N} \mu(d) f(N/d) = \sum_{d \mid N} \mu(d)\sum_{b \mid N/d} M_{q,m,b}(A).
 $$
 Since $f(N/d)$ represents the number of elements in $A$ with orders that are divisors of $N/d$,
 and each such element can be written uniquely as $\alpha^{\frac{(q^m-1)}{N}di}$ for $0 \leq i < N/d$, then
 \begin{align*}
 \sum_{b \mid N/d}M_{q,m,b}(A) &= \left| \left\{ 0 \leq i < N/d \ : \ \alpha^{\frac{(q^m-1)}{N}di} \in A  \right\}    \right|\\
&= \dfrac{N}{d(q^m-1)} \left| \left\{ x \in \Fr^* \ : \ x^{\frac{(q^m-1)}{N}d} \in A  \right\}  \right|.
 \end{align*}
 
 With regards to $P_{q,m,N}(c)$, observe that if we let $A_c := \{ x \in \Fr^* \ : \ \Tmq(x) = c  \}$, then $P_{q,m,N}(c) = M_{q,m,N}(A_c)$.
 Now it remains to obtain the expression for $| \{ x \in \Fr^* \ : \ x^{\frac{q^m-1}{N}d} \in A_c\} | = |\{ x \in \Fr^* \ : \ \Tmq(x^{\frac{q^m-1}{N}d}) = c\} |$, which can be done by applying
 (\ref{eqn: orthogonality}). Indeed, by (\ref{eqn: orthogonality}) and by the transitivity of the trace function,
 \begin{align*}
  \left|\left\{ x \in \Fr^* \ : \ \Tmq\left(x^{\frac{q^m-1}{N}d}\right) = c\right\} \right| &= \sum_{x \in \Fr^*} \dfrac{1}{q}\sum_{a \in \F}\chi_q\left( a\left( \Tmq\left( x^{\frac{q^m-1}{N}d} \right) - c   \right)  \right)\\
  &= \dfrac{1}{q}\sum_{a \in \F} \overline{\chi_q(ac)}    \sum_{x \in \Fr^*}\chi_q\left( \Tmq\left( a x^{\frac{q^m-1}{N}d} \right) \right)\\
  &= \dfrac{1}{q}\sum_{a \in \F} \overline{\chi_q(ac)}    \sum_{x \in \Fr^*} \chi_{q^m}\left( a x^{\frac{q^m-1}{N}d} \right).
 \end{align*}
 Hence the result follows.
\end{proof}

For $k \in \mathbb{Z}$ and $N \mid q^m-1$, we denote
\begin{align*}
\Gamma_k(N) &= \sum_{d \mid N} \mu(d) \eta_k^{\left( \frac{q^m-1}{N}d, q^m  \right)}.
\end{align*}
For $\beta \in \Fr$,
let $W_H(N, \beta)$ denote the Hamming weight of the $n$-tuple, where $n = (q^m-1)/N$, given by
$$
c(N, \beta) := \left( \Tmq\left(\beta \right), \Tmq\left( \beta \alpha^{N}  \right), \ldots, \Tmq\left( \beta \alpha^{(n-1)N}  \right)    \right).
$$
For an integer $t \neq 0$, let $\omega(t)$ be the number of distinct prime divisors of $t$.

The following proposition gives some general identities relating 
$Z_{q,m,N}(c)$, $P_{q,m,N}(c)$, through the M\"{o}bius inversion formula, as well as shows their connection to Hamming weights. 

\begin{prop}\label{prop: order identities}
Let $q$ be a power of a prime, $m$ be a positive integer, $N \mid q^m-1$ and  $k \in \mathbb{Z}$. Then the following identities hold:
\begin{align*}
\sum_{d \mid \rad(N)} \mu(d) \Delta_k\left( \dfrac{\rad(N)}{d}  \right) &= (-1)^{\omega(N)} \eta_k^{(\rad(N),q^m)};\\
\sum_{d \mid N} \Gamma_k(d) &= \eta_k^{\left( \frac{q^m-1}{N}, q^m   \right)};\\
\sum_{d \mid \rad(N)} \mu(d) \Delta_k\left( \dfrac{\rad(N)}{d}  \right) &= (-1)^{\omega(N)} \sum_{d \mid \left( \frac{q^m-1}{\rad(N)}  \right)} \Gamma_k(d);\\
(-1)^{\omega(N)}\sum_{d \mid \rad(N)} \mu(d) f_0 \left( \dfrac{\rad(N)}{d}, \Delta, c  \right) &= q \sum_{d \mid \left(\frac{q^m-1}{\rad(N)} \right)} P_{q,m,d}(c) - \dfrac{q^m-1}{\rad(N)};\\
f_0(N, \Delta, c) &= q \sum_{d \mid N} \mu(d) \sum_{b \mid \left(\frac{q^m-1}{d}\right)} P_{q,m,b}(c) - \dfrac{q^m-1}{N}\phi(N);  \\
Z_{q,m,N}(c) &= \sum_{d \mid N} \mu(d) \sum_{b \mid \left( \frac{q^m-1}{d}  \right)} P_{q,m,b}(c);\\
Z_{q,m,N}(0) &= \dfrac{q^m-1}{N} \phi(N) - \sum_{d \mid N} \mu(d) W_H(d, 1).
\end{align*}
\end{prop}

\begin{proof}
 To prove the first six identities  we use the M\"{o}bius inversion formula together with Lemma \ref{lem: Z formula} and Lemma \ref{lem: order formula}. The last identity follows from the one before it. Indeed, 
 noting that
 $$
 W_H(d, 1) = \sum_{c \in \F^*} \sum_{b \mid \left(\frac{q^m-1}{d} \right)} P_{q,r,b}(c),
 $$
 we get
 \begin{align*}
  \dfrac{q^m-1}{N} \phi(N) - Z_{q,m,N}(0) &= \sum_{c \in \F^*} Z_{q,m,N}(c) = \sum_{d \mid N} \mu(d) \sum_{c \in \F^*} \sum_{b \mid \left(\frac{q^m-1}{d} \right)} P_{q,r,b}(c)\\
  &= 
  \sum_{d \mid N}  \mu(d) W_H(d, 1).
 \end{align*}
\end{proof}

We can obtain a much simpler relation among $ Z_{q,m,N}(0)$ and $P_{q,m,N}(0)$ in the special case when $L_Q \mid N$. 

\begin{lem}\label{lem: relation}
 Let $D$ be the smallest positive divisor of $q-1$ such that $(q-1)/D$ is coprime to $Q$, and let $N \mid q^m-1$ such that $DQ \mid N$.
 Then $ Z_{q,m,N}(0)$ is related to $P_{q,m,N}(0)$ by the equation
 $$
 Z_{q,m,N}(0 ) = \dfrac{q^m-1}{N}  P_{q,m,N}(0) .
 $$
 In particular we have the following relation:
 \begin{align*}
 P_{q,m,q^m-1}(0) &= \dfrac{\phi(q^m-1)}{\phi(DQ)} P_{q,m,DQ}(0).
 \end{align*}
\end{lem}

\begin{proof}
 First note that since $N$ is a multiple of $DQ$, then $(q^m-1)/N$ divides $(q-1)/D$ and hence is coprime to $Q$. Let  $K$ be the largest divisor
 of $N$ that is coprime to $Q$. Then similarly as done in the proof of Lemma \ref{lem: delta sum}, we have, by Lemma \ref{lem: ding2},
 \begin{align*}
  \sum_{d \mid N} \mu(d) \sum_{i=0}^{q-2} \eta_{Qi}^{\left( \frac{q^m-1}{N}d, q^m \right)} 
  &= \sum_{b \mid K} \sum_{d \mid \gcd\left(Q, N\right)} \mu(bd) \sum_{i=0}^{q-2} \eta_{Qi}^{\left( \frac{q^m-1}{N}bd, q^m  \right)}
   = \sum_{b \mid K} \mu(b) \sum_{d \mid \gcd\left(Q, N\right)} \mu(d) \dfrac{(q-1)d}{\frac{q^m-1}{N}bd} \eta_0^{(d,q^m)}\\
   &= \dfrac{N}{Q} \sum_{b \mid K} \dfrac{\mu(b)}{b} \Delta_0(\gcd\left(Q, N\right)) = \dfrac{N \phi(K)}{Q K} \Delta_0(\gcd\left(Q, N\right)).
 \end{align*}
 Then by Lemma \ref{lem: order formula} and using the fact that $\dfrac{\phi(N)}{N} = \dfrac{\phi(K)}{K} \dfrac{\phi(\gcd\left(Q, N\right))}{\gcd\left(Q, N\right)}$, we get
 \begin{align*}
 P_{q,m,N}(0) &= \dfrac{N}{qQ} \left( Q \dfrac{\phi(N)}{N} + \dfrac{\phi(K)}{K} \Delta_0(\gcd\left(Q, N\right))  \right)\\
 &= \dfrac{N\phi(K)}{qQ K} \left( \dfrac{Q}{\gcd\left(Q, N\right)} \phi(\gcd\left(Q, N\right)) + \Delta_0(\gcd\left(Q, N\right))  \right).
\end{align*}
Now the first identity follows from Theorem \ref{lem: zero trace}. 

For the second, Lemma \ref{lem: primitive zero trace} and the first identity gives
\begin{align*}
 \dfrac{(q-1)/D}{\phi((q-1)/D)}P_{q,m,q^m-1}(0) =  Z_{q,m,Q}(0) &= Z_{q,m,DQ}(0) = \dfrac{q-1}{D} P_{q,m,DQ}(0).
\end{align*}
Hence we obtain
$$
P_{q,m,q^m-1}(0) = \phi\left( \dfrac{q-1}{D}  \right)P_{q,m,DQ}(0).
$$
Now the result follows by noticing that, since $DQ$ is coprime to $(q-1)/D$, then 
$$
\phi(q^m-1) = \phi\left( DQ \dfrac{q-1}{D}  \right) = \phi(DQ) \phi\left( \dfrac{q-1}{D} \right).
$$
\end{proof}

Note that $L_Q := DQ$ is the largest divisor of $q^m-1$ with the same radical as that of $Q$. Moreover
whenever $P_{q,m, q^m-1}(0) \neq 0$, the ratio, of the number of primitive elements with zero trace, to the number of elements of order $L_Q$ with trace zero,
is the same as that of the number of primitive elements to the number of elements of order $L_Q$. 

As a consequence of Theorem \ref{thm: cohen} and Lemma \ref{lem: relation} we obtain the following existence result. For this we will need the fact that for any 
$N \mid (q^m-1)$ and any $d \mid N$, 
the set of $N$-free elements is a subset of the set of $d$-free elements in $\mathbb{F}_{q^m}$.

\begin{thm}\label{thm: existence result}
 Let $m > 1$ and $N \mid q^m-1$ such that $L_Q \mid N$, where $L_Q$ is the largest divisor of $q^m-1$ with the same radical as that of $Q = (q^m-1)/(q-1)$. Then there exists an element $\xi \in \mathbb{F}_{q^m}^*$
 of order $N$ satisfying $\Tmq(\xi) = 0$ if and only if $m \neq 2$ and $(q,m) \neq (4,3)$.
\end{thm}

\begin{proof}
 First by Lemma \ref{lem: primitive zero trace} there exists a primitive element with trace $0$ if and only if there exists a $Q$-free element with trace $0$. By Theorem \ref{thm: cohen} this happens
 if and only if $m \neq 2$ and $(q,m) \neq (4,3)$. 
 Since a $(q^m-1)$-free (primitive) element is also $d$-free for any divisor $d$ of $q^m-1$, if follows that $Z_{q, m, N}(0) > 0$ if $m\neq 2$ and $(q,m) \neq (4,3)$;
 by Lemma \ref{lem: relation} we also have $P_{q, m, N}(0) > 0$.
 On the other hand, if $m = 2$ or $(q,m) = (4,3)$, suppose on the contrary that $P_{q,m,N}(0) > 0$. Then by Lemma \ref{lem: relation} we have $Z_{q,m,N}(0) > 0$. Since $Q \mid N$, then $Z_{q,m,Q}(0) > 0$,
 contradicting the fact that $Z_{q,m,Q}(0) = 0$ when $m = 2$ or $(q,m) = (4,3)$. 
\end{proof}

We apply Lemma \ref{lem: relation} together with Corollary \ref{thm: tuxanidy} to obtain the following consequence. 

\begin{thm}\label{thm: order DQ and mersenne}
 Let $p$ be a Mersenne prime and let $Q = (p^4-1)/(p-1) = (p+1)(p^2+1)$. Then the number of non-zero elements in $\mathbb{F}_{p^4}^*$ with order $2Q$ and absolute trace zero
 is $\phi(2Q)/(p+1) = 2\phi(p^2+1)$.
\end{thm}

\begin{proof}
 In this case $Q = (p+1)(p^2+1)$ and so the only prime diving both $p-1$ and $Q$ is $2$. Since $2$ has multiplicity $1$ in the factorization of $p-1$, it follows that $D = 2$,
 where $D$ is as defined in Lemma \ref{lem: relation}. Then by Lemma \ref{lem: relation} and Corollary \ref{thm: tuxanidy} we get
 \begin{align*}
 P_{q,m,2Q}(0) &= \dfrac{\phi(2Q)\left(  \phi(p^4-1) - \phi\left( \frac{p^4-1}{p+1} \right) \right)}{p \phi(p^4-1)} \\
 &= \dfrac{\phi(2Q)}{p} - \dfrac{\phi(2Q) \phi\left( \frac{p^4-1}{p+1} \right)}{p \phi(p^4-1)}.
 \end{align*}
 By Euler's product formula for $\phi$ and using the fact that $2 \mid p-1$ while $p+1$ is a power of $2$, note that
 \begin{align*}
  (p+1)\phi\left( (p-1)(p^2+1) \right) &= (p+1)(p-1)(p^2+1) \prod_{l \mid (p-1)(p^2+1)} \left( 1 - \dfrac{1}{l}  \right) \\
  &= (p+1)(p-1)(p^2+1) \prod_{l \mid (p+1)(p-1)(p^2+1)} \left( 1 - \dfrac{1}{l}  \right)\\
  &= \phi(p^4-1).
 \end{align*}
 Now from the fact that $(p^4-1)/(p+1) = (p-1)(p^2+1)$, the above yields
 \begin{align*}
  P_{q,m,2Q}(0) &= \dfrac{\phi(2Q)}{p} - \dfrac{\phi(2Q) }{p (p+1)}\\
  &= \dfrac{\phi(2Q)}{p+1}.
 \end{align*}
 Remains to note that, since $p$ is a Mersenne prime, 
 $$
 \dfrac{\phi(2Q)}{p+1} = 2(p^2+1) \prod_{l \mid p^2+1} \left(  1 - \dfrac{1}{l} \right) = 2\phi(p^2+1).
 $$
\end{proof}

\section*{Acknowledgments}

We thank the anonymous referee for the many helpful suggestions that improved the presentation of this paper.

\appendix
\section{}

%The following table gives the number 
%$Z_{p,p^4}(p^4-1, 0)$ of primitive elements in $\mathbb{F}_{p^4}$ with absolute zero trace for the first ten Mersenne primes, $p$.
The following table gives the numbers of primitive elements in the quartic extensions of Mersenne prime fields having absolute trace zero, for the first ten Mersenne primes.
See
Corollary \ref{thm: tuxanidy} for an explicit formula.
SAGE software was used for the computations.
%Table 2 gives some examples of pairs $(q,m)$ for which $Z_{q,q^m}(q^m-1, c \neq 0)$ behaves uniformly.

\begin{table}[ht]
 \centering
 \caption{The number of primitive elements of $\mathbb{F}_{p^4}$ with absolute zero trace for the first ten Mersenne primes $p$}
 \begin{tabular}{|c|c|}
   \hline
   Mersenne prime, $p$ & $Z_{p,4,p^4-1}(0)$\\
   \hline
   $2^2-1$ & $8$ \\
   \hline
   $2^3-1$ & $80$ \\
   \hline
   $2^5-1$ & $6,912$\\
   \hline
   $2^7-1$ & $464,256$\\
   \hline
   $2^{13}-1$ & $111,974,400,000$\\
   \hline
   $2^{17}-1$ & $519,390,596,431,872$\\
   \hline
   $2^{19}-1$ & $30,572,599,504,748,544$\\
   \hline
   $2^{31}-1$ & $1,968,482,608,781,191,263,129,600,000$\\
   \hline
   $2^{61}-1$ & $2,159,465,982,279,294,537,199,679,191,$\\
              & $374,585,254,935,265,280,000,000,000$\\
   \hline
   $2^{89}-1$ & $51,505,739,520,752,637,174,787,391,794,396,705,748,179,$\\
              & $291,647,742,969,497,437,393,928,825,245,616,046,080$\\
   \hline
   \end{tabular}
   \end{table}
 
 \iffalse
\begin{table}[ht]
 \centering
 \caption{Examples of pairs $(q,m)$ for which $\rad(q^m-1) = \rad((q^m-1)/(q-1))$ whence $Z_{q,q^m}(q^m-1, c \neq 0)$ behaves uniformly}
 \begin{tabular}{|c|c|c|c|c|c|c|c|c|c|c|c|}
   \hline
   $q$ & 3 & 8 & 13 & 19 & 25 & 29 & 32 & 37 & 41 & 47\\
   \hline
   $m$ & 50 & 49 & 48 & 42 & 36 & 28 & 31 & 30 & 40 & 46\\
   \hline
   \end{tabular}
   \end{table}

\fi

\end{document}